\documentclass[11pt]{article}
\usepackage{booktabs}
\usepackage{caption}
\usepackage{mathrsfs}
\usepackage{amsmath}
\usepackage{amsfonts,amsthm,amssymb,mathrsfs,bbding}
\usepackage{graphics,multicol}
\usepackage{graphicx}
\usepackage{color}
\usepackage{enumerate}
\usepackage{caption}
\usepackage{rotating}
\usepackage{lscape}
\usepackage{longtable}
\allowdisplaybreaks[4]
\usepackage{tabularx}
\usepackage[colorlinks=true,anchorcolor=blue,filecolor=blue,linkcolor=blue,urlcolor=blue,citecolor=blue]{hyperref}
\usepackage{extarrows}
\usepackage{cite}
\usepackage{latexsym,bm}
\usepackage{mathtools}
\evensidemargin=\oddsidemargin\topmargin=-1.5cm

\newtheorem{thm}{Theorem}[section]

\newtheorem{lem}{Lemma}[section]

\newtheorem{claim}{Claim}[section]
\theoremstyle{definition}

\addtocounter{section}{0}
\begin{document}

	\title{\bf On the spectral extremal problem of planar graphs}
	\author{{Xiaolong Wang, Xueyi Huang, and Huiqiu Lin \setcounter{footnote}{-1}\footnote{\emph{E-mail address:}
	xlwangmath@163.com (X. Wang),
	huangxymath@163.com (X. Huang),  huiqiulin@126.com (H. Lin)}}\\
		\small School of Mathematics, East China University of Science and Technology,\\
		\small Shanghai 200237, China}

	\date{}
	\maketitle
	{\flushleft\large\bf Abstract}
The spectral extremal problem of planar graphs has aroused a lot of interest over the past three decades. In 1991, Boots and Royle [Geogr. Anal. 23(3) (1991) 276--282] (and Cao and Vince [Linear Algebra Appl. 187
(1993) 251--257] independently) conjectured that $K_2 + P_{n-2}$ is the unique graph attaining the maximum spectral radius among all planar graphs on $n$ vertices, where $K_2 + P_{n-2}$ is the graph obtained from $K_2\cup P_{n-2}$ by adding all possible edges between $K_2$ and $P_{n-2}$.  In 2017, Tait and Tobin [J. Combin. Theory Ser. B 126 (2017) 137--161] confirmed this conjecture for all sufficiently large $n$. In this paper, we consider the spectral extremal problem for planar graphs without specified subgraphs. For a fixed graph $F$, let $\mathrm{SPEX}_{\mathcal{P}}(n,F)$ denote the set of graphs attaining the maximum spectral radius among all  $F$-free planar graphs on $n$ vertices.  We describe a rough sturcture for the connected extremal graphs in
$\mathrm{SPEX}_{\mathcal{P}}(n,F)$ when $F$ is a planar graph not contained in $K_{2,n-2}$. As applications, we determine the extremal graphs in $\mathrm{SPEX}_{\mathcal{P}}(n,W_k)$, $\mathrm{SPEX}_{\mathcal{P}}(n,F_k)$ and $\mathrm{SPEX}_{\mathcal{P}}(n,(k+1)K_2)$ for all sufficiently large $n$, where $W_k$, $F_k$ and $(k+1)K_2$ are the wheel graph of order $k$, the friendship graph of order $2k+1$ and the disjoint union of $k+1$ copies of $K_2$, respectively. 
	
	\begin{flushleft}
		
\bigskip

	\end{flushleft}

	\section{Introduction}
	
Let $\mathcal{F}$ be a family of graphs. A graph is said to be $\mathcal{F}$-free if it has no subgraph isomorphic to some $F\in\mathcal{F}$. In particular, if $\mathcal{F}=\{F\}$, then we write $F$-free intead of $\mathcal{F}$-free. The \textit{Tur\'{a}n number} of $\mathcal{F}$, denoted by $\mathrm{ex}(n,\mathcal{F})$, is the maximum number of edges in an $\mathcal{F}$-free graph on  $n$ vertices.  In extremal graph theory, the Tur\'{a}n problem has aroused a lot of interest. Two earliest results on Tur\'{a}n problem are the Mantel's theorem and the Tur\'{a}n's theorem. In 1907, Mantel \cite{Man07} determined the Tur\'{a}n number of $K_3$, which is attained by balanced complete bipartite graphs. As an extenstion of Mantel's theorem, Tur\'{a}n \cite{Tur41}  determined the Tur\'{a}n number of $K_{r+1}$ for every integer $r\geq 2$, which is attained by balanced complete $r$-partite graphs. Over the past half century, a large quantity of work has been carried out in  the Tur\'{a}n problems of various graphs, such as cycles \cite{FG15,Ver00}, wheels  \cite{Yua21}, friendship graphs \cite{EFGG95}, and so on. Nethertheless, the Tur\'{a}n problem  for even cycles is still far from resolved. For more results on Tur\'{a}n problem, we refer the reader to the survery paper \cite{FS00}.

In 2015, Dowden \cite{Dow16} initiated the study of planar Tur\'{a}n problem. The \textit{planar Tur\'{a}n number} of  $\mathcal{F}$, denoted by $\mathrm{ex}_{\mathcal{P}}(n,\mathcal{F})$, is the maximum number of edges in an $\mathcal{F}$-free planar graph on $n$ vertices. It was shown in   \cite{Dow16} that $\mathrm{ex}_{\mathcal{P}}(n,K_3)=2n-4$ and $\mathrm{ex}_{\mathcal{P}}(n,K_r)=3n-6$ for all $n\geq 6$ and $r\ge 4$. Since  there are arbitrarily large triangulations with maximum degree $6$, the planar Tur\'{a}n problem is trivial for graphs with maximum degree more than $6$. For this reson, the next natural type of graphs considered are cycles and others variations. Dowden \cite{Dow16} proved that $\mathrm{ex}_{\mathcal{P}}(n,C_4)\le \frac{15(n-2)}{7}$ for  $n\ge 4$ and  $\mathrm{ex}_{\mathcal{P}}(n,C_5)\le \frac{12n-33}{5}$ for $n\ge 11$, and obtained obtained infinitely many extremal
graphs attaining these upper bounds. Ghosh, Gy\H{o}ri, Martin, Paulos and  Xiao \cite{GGMPX20} proved that $\mathrm{ex}_{\mathcal{P}}(n,C_6)\le \frac{5n}{2}-7$ for $n\ge 18$. Lan, Shi and Song \cite{LSS19} present a sharp upper bound of $\mathrm{ex}_{\mathcal{P}}(n,\mathcal{C}_k+e)$ for $k\in \{4,5\}$ and an upper bound of $\mathrm{ex}_{\mathcal{P}}(n,\mathcal{C}_6+e)$, where $\mathcal{C}_k+e$ is the family of graphs obtained from a cycle $C_k$ by linking two nonadjacent vertices (via $e$) in the cycle. Recently, Fang, Wang and Zhai \cite{FWZ22} obtained sharp bounds of the planar Tur\'{a}n number of $k$-fans and friendship graphs for all non-trivial cases.

Let $G$ be a graph, and let $A(G)$ denote the adjacency matrix of $G$. The \textit{spectral radius} of $G$, denoted by $\rho(G)$, is the largest eigenvalue of $A(G)$. Analogous to the Tur\'{a}n problem, let $\mathrm{spex}(n, \mathcal{F})$ denote the maximum spectral radius of any $\mathcal{F}$-free graph on $n$ vertices, and  $\mathrm{SPEX}(n,\mathcal{F})$ denote the set of extremal graphs with respect to  $\mathrm{spex}(n,\mathcal{F})$. In 2010, Nikiforov \cite{Nik10} pioneered the systematic study of spectral Tur\'{a}n problems, although there are several earlier results on this topic. In recent years, the spectral Tur\'{a}n problem has become very popular, and a lot of attention has been paid on determining $\mathrm{spex}(n, \mathcal{F})$ and characterizing  $\mathrm{SPEX}(n, \mathcal{F})$ for various families of graphs \cite{CDT23,CDT22,CDT221,CFTZ20,CLZ24,LL24,LLT06,LP23,Nik08,WKX23,YWZ12,ZL20,ZL22,ZL23,ZLX22,ZW12,ZW20,ZW23}. Very recently, Cioab\u{a}, Desai and Tait \cite{CDT22} made a breakthrough on spectral Tur\'{a}n problems. They determined the exact value of $\mathrm{spex}(n, C_{2k})$ and characterized the extremal graphs in $\mathrm{SPEX}(n, C_{2k})$ for any $k\geq 3$ and sufficiently large $n$, which confirmed  a famous conjecture of Nikiforov in 2010 \cite{Nik10}. 

Spectral Tur\'{a}n problems belong to a broader framework of problems called Brualdi-Solheid problems \cite{BS86} that investigate the maximum spectral radius among all graphs belonging to a specified family of graphs. For planar graphs, the study of  Brualdi-Solheid problems has achieved abundant results. Let $G$ be a planar graph of order $n$. In 1988, Yuan \cite{Yua88} proved that $\rho(G) \leq \sqrt{5n-11}$. In 1993, Cao and Vince \cite{CV93} improved the upper bound to  $4+\sqrt{3n-9}$. Later, Yuan \cite{Yua95} improved the upper bound to $2\sqrt{2}+\sqrt{3n-\frac{15}{2}}$. In 2000,  Ellingham and Zha \cite{EZ00} improved the  upper bound to $2+\sqrt{2n-6}$. Additionally, it was conjectured by Boots and Royle \cite{BR91}  in 1991 (and independently by Cao and Vince \cite{CV93} in 1993) that $K_2 + P_{n-2}$  is the unique graph attaining the maximum spectral radius among all planar graphs of order $n$.  Until  2017, Tait and Tobin \cite{TT17} confirmed the conjecture for all sufficiently large $n$. Among other things,  Dvo\v{r}\'{a}k and Mohar \cite{DM10} found an upper bound on the spectral radius of planar graphs with a given maximum degree. For more results on the spectral radius of planar graphs, we refer the reader to \cite{LN21}.   

As an analogue of the planar Tur\'{a}n problem, it is natural to consider the spectral Tur\'{a}n problem within  planar graphs (called \textit{planar spectral Tur\'{a}n problem}).  Let $\mathrm{spex}_{\mathcal{P}}(n,\mathcal{F})$ denote the maximum spectral radius of the adjacency matrix of any $\mathcal{F}$-free planar graphs on $n$ vertices, and  $\mathrm{SPEX}_{\mathcal{P}}(n,\mathcal{F})$ denote the set of extremal graphs with respect to  $\mathrm{spex}_{\mathcal{P}}(n,\mathcal{F})$. In 2022, Zhai and Liu \cite{ZHL22} characterized the extremal graphs in $\mathrm{SPEX}_{\mathcal{P}}(n,\mathcal{F})$ when $\mathcal{F}$ is the family of $k$ edge-disjoint cycles. Very recently, Fang, Lin and Shi \cite{FLS23} determined the extremal graphs in $\mathrm{SPEX}_{\mathcal{P}}(n,tC_\ell)$ and $\mathrm{SPEX}_{\mathcal{P}}(n,t\mathcal{C})$, where $tC_\ell$ is the
disjoint union of $t$ copies of $\ell$-cycles, and $t\mathcal{C}$ is the family of $t$ vertex-disjoint cycles without length restriction. In this paper, we provide a structural theorem for connected extremal graphs in 
$\mathrm{SPEX}_{\mathcal{P}}(n,F)$ under the condition that $F$ is a planar graph not contained in $K_{2,n-2}$ and $n$ is sufficient large relative to the order of $F$, which extends a result of Fang, Lin and Shi (cf. \cite[Theorem 1.1]{FLS23}).

\begin{thm}\label{thm::1}
Let $F$ be a planar graph not contained in $K_{2,n-2}$ where
$n\geq \max\{2.67\times9^{17},\frac{10}{9}|V(F)|\}$. Suppose that $G$ is a connected extremal graph in $\mathrm{SPEX}_{\mathcal{P}}(n,F)$ and $X=(x_v:v\in V(G))^T$ is the positive eigenvector of $\rho:=\rho(G)$ with $\max_{v\in V(G)} x_v=1$. Then the following two statements hold.
\begin{enumerate}[(i)]
    \item There exist two vertices $u^\prime,u^{\prime\prime}\in V(G)$ such that $R:=N_{G}(u^\prime)\cap N_{G}(u^{\prime\prime})=V(G)\setminus\{u^{\prime},u^{\prime\prime}\}$ and $x_{u^{\prime}}=x_{u^{\prime\prime}}=1$. In particular, $G$ contains a copy of $K_{2,n-2}$.
    
    \item The subgraph $G[R]$ of $G$ induced by $R$ is a disjoint union of some paths and cycles. Moreover, if $G[R]$ contains a cycle then it is exactly a cycle, i.e., $G[R]\cong C_{n-2}$, and if $u^{\prime}u^{\prime\prime} \in E(G)$ then $G[R]$ is a disjoint union of some paths. 
\end{enumerate}
\end{thm} 

The \textit{join} of two graphs $G$ and $H$, denoted by $G+H$, is the graph obtained from $G\cup H$ by adding all possible edges between $G$ and $G$. The \textit{wheel graph} of order $k$ and the \textit{friendship graph} of order $2k+1$ are then defined as $W_k=K_1+C_{k-1}$ and $F_k=K_1+kK_2$, respectively.
As applications of Theorem \ref{thm::1}, we consider the planar spectral Tur\'{a}n problem for wheel graphs, friendship graphs, and independent edges, respectively. According to a result of Nikiforov \cite{Nik09}, the extremal graph with respect to $\mathrm{spex}(n,W_{2\ell})$ for sufficiently large $n$ is exactly the Tur\'{a}n
graph with three parts. For odd wheel graphs, Cioab\u{a}, Desai and Tait \cite{CDT221} determined the structure
of the extremal graphs with respect to $\mathrm{spex}(n,W_{2\ell+1})$ for all $\ell\geq 2$, $\ell\notin\{4,5\}$ and sufficiently large $n$. With regard to planar graphs, we obtain the following result.
\begin{thm}\label{thm::2}
Let $k$ and $n$ be integers with $k\geq 3$ and  $n\geq \max\{2.67\times9^{17},10.2\times2^{k-4}+2\}$. 
Then $\mathrm{SPEX}_{\mathcal{P}}(n,W_k)=\{W_{n,k}\}$, where
\begin{equation*}
          W_{n,k}:= 
          \begin{cases}
          K_{2,n-2},& \mbox{if $k=3$},\\
          2K_1+C_{n-2},& \mbox{if $k=4$},\\
          K_{2}+\left(\lfloor\frac{n-2}{k-3}\rfloor P_{k-3} \cup P_{n-2-(k-3)\cdot\lfloor\frac{n-2}{k-3}\rfloor}\right), & \mbox{if $k\geq 5$}.
          \end{cases}
     \end{equation*} 
\end{thm}

In 2020, Cioab\u{a}, Feng, Tait and Zhang \cite{CFTZ20} proved that every extremal graph in $\mathrm{SPEX}(n,F_k)$ is exactly an extremal graph with respect to $\mathrm{ex}(n,F_k)$ for all sufficiently large $n$. Based on this result, Zhai, Liu and Xue \cite{ZLX22} identified the unique extremal graph in  $\mathrm{SPEX}(n,F_k)$ for all sufficiently large $n$. 
For planar graphs, we prove that the extremal graph is also unique.

\begin{thm}\label{thm::3}
Let $k$ and $n$ be integers with $k\geq 1$ and $n\geq \max\{2.67\times9^{17},10.2\times2^{2k-4}+2\}$. Then  $\mathrm{SPEX}_{\mathcal{P}}(n,F_k)=\{F_{n,k}\}$, where
\begin{equation*}
          F_{n,k}:=  \begin{cases}K_{2,n-2}, & \mbox{if $k=1$}, 
          \\ K_{2}+(n-2)K_1, & \mbox{if $k=2$},
          \\ K_{2}+(P_{2k-3} \cup (n-2k+1)K_1), & \mbox{if $k\geq 3$}.
          \end{cases}
     \end{equation*}
\end{thm}

The \textit{matching number} of a graph $G$ is maximum size of a subset of $E(G)$ that contains only independent edges. In 2007, Feng, Yu and Zhang \cite{FYZ07} characterized the extremal graphs attaining the maximum spectral radius among all graph on $n$ vertices with given matching number.  For planar graphs, we also obtain the extremal graphs.

\begin{thm}\label{thm::4}
Let $k$ and $n$ be integers with $k\geq 1$ and $n\geq N+\frac{3}{2}\sqrt{2N-6}$ where $N=\max\{2.67\times9^{17},10.2\times2^{2k-4}+2\}$. Then  $\mathrm{SPEX}_{\mathcal{P}}(n,(k+1)K_2)=\{M_{n,k}\}$, where
$$ 
M_{n,k}=\begin{cases}
K_{1,n-1}, &  \mbox{if $k=1$}, \\ 
K_2+(n-2)K_1, &  \mbox{if $k=2$}, \\ K_2+(P_{2k-3}\cup (n-2k+1)K_1), & \mbox{if $k\geq 3$}.
\end{cases}   
$$
In particular, $M_{n,k}$ is the  unique graph attaining the maximum spectral radius among all planar graphs of order $n$ with matching number $k$.
\end{thm}

\section{Key Lemmas}\label{sec::2}
In this section, we list a series of lemmas, which focuses on the structural properties of a special class of graphs, namely the connected planar graphs on $n$ vertices  with spectral radius not less than $\sqrt{2n-4}$. These properties play a key role in the proof of our main results.

First of all, we recall a classic result of Ellingham and Zhang \cite{EZ00} regarding the upper bound of the spectral radius of planar graphs.  
\begin{lem}(Ellingham and Zha \cite{EZ00}\label{lem::1})
    Let $G$ be a planar graph with $n\geq 3$. Then $\rho(G)\leq 2+\sqrt{2n-6}$.
\end{lem}

Before going further, we introduce some notations and symbols. Let $G$ be a  graph with vertex set $V(G)$ and edge set $E(G)$. For any $v\in V(G)$, let $N_{i}(v)$ denote the set of vertices at distance $i$ from $u$ in $G$. In particular, we write $N_G(v)$ or $N(v)$ instead of $N_1(v)$, and denote by  $d_G(v):=|N_{G}(v)|$. Also, for any subset $S\subseteq V(G)$, let $N_S(v)$ denote the set of vertices in $S$ that are adjacent to  $v$.  For any two disjoint subset $S, T \subseteq V(G)$, let $G[S]$ denote the subgrah of $G$ induced by $S$, and let $G[S, T]$ denote the bipartite subgraph of $G$ with vertex set $S \cup T$ consisting of all edges between $S$ and $T$ in $G$. Set $e(S)=|E(G[S])|$ and $e(S, T)=|E(G[S, T])|$. For any planar graph $G$, it is known that  
\begin{equation}\label{equ::2}
e(S) \leq 3|S|-6 \text { and } e(S, T) \leq 2(|S|+|T|)-4,
\end{equation}
where $S$ and $T$ are disjoint subsets of $V(G)$.

Assume that $G$ is a connected planar graph on $n$ vertices with $n\geq 2.67\times 9^{17}$ and $\rho:=\rho(G)\geq\sqrt{2n-4}$. In what follows, we shall present some structural properties possessed by  the graph $G$ (see Lemmas \ref{lem::2}--\ref{lem::8} below).

Let  $X=(x_{v}:v\in V(G))^{T}$ be the positive eigenvector of  $\rho:=\rho(G)$ with $\max_{v\in V(G)} x_{v}=1$.   For any real number $\lambda \geq \frac{1}{9^{3}}$, we define $$L^{\lambda}=\left\{u \in V(G) \mid x_{u} \geq \frac{1}{9^{3} \lambda}\right\}.$$ 
	\begin{lem}\label{lem::2}
		$|L^{\lambda}| \leq \frac{\lambda n}{9^{5}}$.	
	\end{lem}
        \begin{proof}
By assumption,  $\rho \geq \sqrt{2 n-4}$, and hence 
$$
\frac{\sqrt{2 n-4}}{9^{3} \lambda} \leq \rho x_{u}=\sum_{v \in N_{G}(u)} x_{v} \leq d_{G}(u)
$$
for each $u \in L^{\lambda}$. Summing this inequality over all vertices $u \in L^{\lambda}$, we obtain

$$
\frac{\sqrt{2 n-4}}{9^{3} \lambda}\cdot |L^{\lambda}|  \leq \sum_{u \in L^{\lambda}} d_{G}(u) \leq \sum_{u \in V(G)} d_{G}(u) \leq 2(3 n-6),
$$
which implies  that $|L^{\lambda}| \leq 3 \times 9^{3} \lambda \sqrt{2 n-4} \leq \frac{\lambda n}{9^{5}}$ as $n \geq 2.67 \times 9^{17}$.
        \end{proof}

    \begin{lem}\label{lem::3}
$|L^{1}| < 6 \times 9^{4}$.
    \end{lem}

  \begin{proof}
Let $u$ be an arbitrary vertex of $G$. For convenience, we denote $ L_{i}^{\lambda}(u)=N_{i}(u) \cap L^{\lambda}$ and $\overline{L_{i}^{\lambda}}(u)=N_{i}(u) \setminus L^{\lambda}$. Recall that $\rho \geq \sqrt{2 n-4}$. Then
\begin{equation}\label{equ::3}
\begin{aligned}
(2 n-4) x_{u} &\leq \rho^{2} x_{u}=d_{G}(u) x_{u}+\sum_{v \in N_{1}(u)} \sum_{w \in N_{1}(v) \setminus \{u\}} x_{w}\\
&\leq d_{G}(u) x_{u}+\sum_{v \in N_{1}(u)} \sum_{w \in L_{1}^{\lambda}(u) \cup L_{2}^{\lambda}(u)} x_{w} +\sum_{v \in N_{1}(u)} \sum_{w\in \overline{L_{1}^{\lambda}}(u) \cup \overline{L_{2}^{\lambda}}(u)} x_{w},
\end{aligned}
\end{equation}
where the last inequality follows from the fact that $N_{1}(v) \setminus \{u\} \subseteq N_{1}(u) \cup N_{2}(u)=L_{1}^{\lambda}(u) \cup L_{2}^{\lambda}(u) \cup \overline{L_{1}^{\lambda}}(u) \cup \overline{L_{2}^{\lambda}}(u)$. 

Note that $N_{1}(u)=L_{1}^{\lambda}(u) \cup \overline{L_{1}^{\lambda}}(u)$ and $x_{w} \leq 1$ for each $w \in L_{1}^{\lambda}(u) \cup L_{2}^{\lambda}(u)$. We have
\begin{equation}\label{equ::4}
\begin{aligned}
\sum_{v \in N_{1}(u)} \sum_{w \in L_{1}^{\lambda}(u) \cup L_{2}^{\lambda}(u)} x_{w}&=\sum_{v \in L_{1}^{\lambda}(u)} \sum_{w \in L_{1}^{\lambda}(u) \cup L_{2}^{\lambda}(u)} x_{w}
+\sum_{v \in \overline{L_{1}^{\lambda}}(u)} \sum_{w \in L_{1}^{\lambda}(u) \cup L_{2}^{\lambda}(u)} x_{w}\\
&\leq\left(2 e(L_{1}^{\lambda}(u))+e(L_{1}^{\lambda}(u), L_{2}^{\lambda}(u))\right)+\sum_{v \in \overline{L_{1}^{\lambda}}(u)} \sum_{w \in L_{1}^{\lambda}(u) \cup L_{2}^{\lambda}(u)} x_{w}
\end{aligned}
\end{equation}
Recall that  $L_{1}^{\lambda}(u) \cup L_{2}^{\lambda}(u) \subseteq L^{\lambda}$, and  $\left|L^{\lambda}\right| \leq \frac{\lambda n}{9^{5}}$ by Lemma \ref{lem::2}. Then from \eqref{equ::2} we obtain
\begin{equation}\label{equ::5}
\begin{aligned}
2 e(L_{1}^{\lambda}(u))+e(L_{1}^{\lambda}(u), L_{2}^{\lambda}(u)) &\leq 2(3|L_{1}^{\lambda}(u)|-6)+(2(|L_{1}^{\lambda}(u)|+|L_{2}^{\lambda}(u)|)-4)\\
&<8|L^{\lambda}|\leq \frac{8 \lambda n}{9^{5}}.
\end{aligned}
\end{equation}
Also note that $x_{w} < \frac{1}{9^{3} \lambda}$ for each $w \in \overline{L_{1}^{\lambda}}(u) \cup \overline{L_{2}^{\lambda}}(u)$. Then

\begin{equation}\label{equ::6}
\begin{aligned}
    \sum_{v \in N_{1}(u)} \sum_{w \in \overline{L_{1}^{\lambda}}(u) \cup \overline{L_{2}^{\lambda}}(u)} x_{w}&=\sum_{v \in L_{1}^{\lambda}(u) \cup \overline{L_{1}^{\lambda}}(u)} \sum_{w \in \overline{L_{1}^{\lambda}}(u) \cup \overline{L_{2}^{\lambda}}(u)} x_{w}\\
    &\leq\left(e(L_{1}^{\lambda}(u), \overline{L_{1}^{\lambda}}(u) \cup \overline{L_{2}^{\lambda}}(u))+2 e(\overline{L_{1}^{\lambda}}(u))+e(\overline{L_{1}^{\lambda}}(u), \overline{L_{2}^{\lambda}}(u))\right)\cdot \frac{1}{9^{3} \lambda}\\
    &\leq \min\left\{2|{L_{1}^{\lambda}}(u)|+10|\overline{L_{1}^{\lambda}}(u)|+4|\overline{L_{2}^{\lambda}}(u)|,6n\right\}\cdot \frac{1}{9^{3} \lambda},
\end{aligned}
\end{equation}
where the last inequality follows from 
 $$e(L_{1}^{\lambda}(u), \overline{L_{1}^{\lambda}}(u) \cup \overline{L_{2}^{\lambda}}(u))+2 e(\overline{L_{1}^{\lambda}}(u))+e(\overline{L_{1}^{\lambda}}(u), \overline{L_{2}^{\lambda}}(u)) \leq 2 e(G)<6 n$$ 
and 
$$
\begin{aligned}
   &~~~~e(L_{1}^{\lambda}(u), \overline{L_{1}^{\lambda}}(u) \cup \overline{L_{2}^{\lambda}}(u))+2 e(\overline{L_{1}^{\lambda}}(u))+e(\overline{L_{1}^{\lambda}}(u), \overline{L_{2}^{\lambda}}(u))\\
   &< 2(|{L_{1}^{\lambda}}(u)|+|\overline{L_{1}^{\lambda}}(u)|+|\overline{L_{2}^{\lambda}}(u)|)+6|\overline{L_{1}^{\lambda}}(u)|+2(|\overline{L_{1}^{\lambda}}(u)|+|\overline{L_{2}^{\lambda}}(u)|)\\
    & =2|{L_{1}^{\lambda}}(u)|+10|\overline{L_{1}^{\lambda}}(u)|+4|\overline{L_{2}^{\lambda}}(u)|.
\end{aligned}     
$$
Combining \eqref{equ::3}--\eqref{equ::6}, we obtain
\begin{equation}\label{equ::7}
\begin{aligned}
    (2 n-4) x_{u}&<d_{G}(u) x_{u}+\sum_{v \in \overline{L_{1}^{\lambda}}(u)} \sum_{w \in L_{1}^{\lambda}(u) \cup L_{2}^{\lambda}(u)} x_{w}\\
    &~~~~+\min\left\{\left(\frac{8 \lambda}{9}+\frac{54}{\lambda}\right)n,\frac{8n \lambda}{9}+\frac{9(2|{L_{1}^{\lambda}}(u)|+10|\overline{L_{1}^{\lambda}}(u)|+4|\overline{L_{2}^{\lambda}}(u)|)}{\lambda}\right\}\cdot \frac{1}{9^{4}}.
\end{aligned}
\end{equation}

Now we prove that $d_{G}(u) \geq \frac{n}{9^{4}}$ for each $u \in L^{1}$. Suppose to the contrary that there exists a vertex $\widetilde{u} \in L^{1}$ such that $d_{G}(\widetilde{u})<\frac{n}{9^{4}}$. Note that $x_{\widetilde{u}} \geq \frac{1}{9^{3}}$ because $\widetilde{u} \in L^{1}$. Take $u=\widetilde{u}$ and $\lambda=9$. Since $N_{1}(\widetilde u)={L_{1}^{9}}(\widetilde u) \cup \overline{L_{1}^{9}}(\widetilde u)$, we have
$$
    2|{L_{1}^{9}}(\widetilde u)|+10|\overline{L_{1}^{9}}(\widetilde u)|+4|\overline{L_{2}^{9}}(\widetilde u)|
    \leq 10d_{G}(\widetilde{u})+4|\overline{L_{2}^{9}}(\widetilde u)|<\frac{10n}{9^{4}}+4n.
$$
and it follows from \eqref{equ::7} that
\begin{equation}\label{equ::8}
\frac{2 n-4}{9^{3}}<d_{G}(\widetilde{u}) x_{\widetilde{u}}+\sum_{v \in \overline{L_{1}^{9}}(\widetilde u)} \sum_{w \in L_{1}^{9}(\widetilde u )\cup L_{2}^{9}(\widetilde u)} x_{w}+\left(12+\frac{10}{9^{4}}\right)\cdot\frac{ n}{9^{4}}.
\end{equation}
Moreover,  since $\left|N_{1}(\widetilde u)\right|=d_{G}(\widetilde{u})<\frac{n}{9^{4}}$ and $\left|L^{9}\right| \leq \frac{n}{9^{4}}$, we have
$$
\begin{aligned}
d_{G}(\widetilde{u}) x_{\widetilde{u}}+\sum_{v \in \overline{L_{1}^{9}}(\widetilde u)} \sum_{w \in L_{1}^{9}(\widetilde u )\cup L_{2}^{9}(\widetilde u)} x_{w}&\leq d_{G}(\widetilde{u})+e(\overline{L_{1}^{9}}(\widetilde u), L_{1}^{9}(\widetilde u) \cup L_{2}^{9}(\widetilde u))\\
&<d_{G}(\widetilde{u})+2(|\overline{L_{1}^{9}}(\widetilde u)|+|L_{1}^{9}(\widetilde u)|+|L_{2}^{9}(\widetilde u)|)\\
&\leq 3d_{G}(\widetilde{u})+2|L^{9}|\\
 &\leq \frac{5 n}{9^{4}},
\end{aligned}
$$
where the second inequality follows from \eqref{equ::2}. Then \eqref{equ::8} implies that
$$
\frac{2 n-4}{9^{3}}<\frac{5 n}{9^{4}}+\left(12+\frac{10}{9^{4}}\right)\cdot\frac{ n}{9^{4}}=\left(17+\frac{10}{9^{4}}\right)\cdot\frac{ n}{9^{4}},
$$
which is a contradiction. Therefore, $d_{G}(u) \geq \frac{n}{9^{4}}$ for all $u \in L^{1}$. Summing this inequality over all vertices $u \in L^{1}$, we obtain

$$
\left|L^{1}\right| \frac{n}{9^{4}} \leq \sum_{u \in L^{1}} d_{G}(u) \leq 2 e(G) < 6 n,
$$
and hence $\left|L^{1}\right| < 6 \times 9^{4}$. The result follows.
    \end{proof}
For convenience, we use $L$, $L_{i}(u)$ and $\overline{L_{i}}(u)$ instead of $L^{1}, N_{i}(u) \cap L^{1}$ and $N_{i}(u) \setminus  L^{1}$, respectively. 
   
    \begin{lem}\label{lem::4}
   	For every $u \in L$, we have $d_{G}(u) \geq\left(x_{u}-\frac{4}{729}\right) n$.
    \end{lem}

\begin{proof}
Let ${\overline{L_{1}}}^{\prime}(u)$ be the subset of $\overline{L_{1}}(u)$ in which each vertex has at least two neighbors in $L_{1}(u) \cup L_{2}(u)$. We claim that $|{\overline{L_{1}}}^{\prime}(u)| \leq |L_{1}(u) \cup L_{2}(u)|^{2}$. If $|L_{1}(u) \cup L_{2}(u)|=1$, then ${\overline{L_{1}}}^{\prime}(u)=\emptyset$, as desired. Now assume that $|L_{1}(u) \cup L_{2}(u)| \geq 2$. Suppose to the contrary that $|\overline{L_{1}}^{\prime}(u)|>|L_{1}(u) \cup L_{2}(u)|^{2}$. Since there are only $\tbinom{|L_{1}(u) \cup L_{2}(u)|}{2}$ options for vertices in ${\overline{L_{1}}}^{\prime}(u)$ to choose  two neighbors from $L_{1}(u) \cup L_{2}(u)$, we can find two vertices in $L_{1}(u) \cup L_{2}(u)$ with at least $\lceil{|{\overline{L_{1}}}^{\prime}(u)|/\tbinom{|L_{1}(u) \cup L_{2}(u)|}{2}}\rceil\geq 3$ common neighbors in ${\overline{L_{1}}}^{\prime}(u)$. Also note that $u \notin L_{1}(u) \cup L_{2}(u)$ and ${\overline{L_{1}}}^{\prime}(u) \subseteq \overline{L_{1}}(u) \subseteq N_{1}(u)$. Then we see that $G$ contains a copy of $K_{3,3}$, which is impossible because $G$ is planar. Therefore, $|{\overline{L_{1}}}^{\prime}(u)| \leq |L_{1}(u) \cup L_{2}(u)|^{2}$, and
\begin{equation}\label{equ::8_1}
\begin{aligned}
e(\overline{L_{1}}(u), L_{1}(u) \cup L_{2}(u)) &=e(\overline{L_{1}}(u) \setminus  {\overline{{L_{1}}}^{\prime}}(u), L_{1}(u) \cup L_{2}(u))+e({\overline{{L_{1}}}^{\prime}}(u), L_{1}(u) \cup L_{2}(u))\\
&\leq|\overline{L_{1}}(u) \setminus  {\overline{{L_{1}}}^{\prime}}(u)|+|L_{1}(u) \cup L_{2}(u)|\cdot|{\overline{L_{1}}}^{\prime}(u)| \\
&\leq d_{G}(u)+(6 \times 9^{4})^{3} \\
&\leq d_{G}(u)+\frac{n}{729},
\end{aligned}
\end{equation}
where the penultimate inequality follows from $\overline{L_{1}}(u)\subseteq N_1(u)$ and  $|L_{1}(u) \cup L_{2}(u)| \leq|L| \leq 6 \times 9^{4}$ (by Lemma \ref{lem::3}), and the last inequality holds as $n \geq 2.67 \times 9^{17}$. Putting $\lambda=1$ in \eqref{equ::7} and combining it with \eqref{equ::8_1}, we obtain
$$
(2 n-4) x_{u} \leq d_{G}(u)+\left(d_{G}(u)+\frac{n}{729}\right)+\frac{55 n}{9^{4}},
$$
and therefore,  
$$d_{G}(u) \geq(n-2) x_{u}-\frac{64 n}{2 \times 9^{4}} \geq\left(x_{u}-\frac{4}{729}\right) n,$$
as desired.
\end{proof}

Take $u^{\prime}\in V(G)$ such that $x_{u^{\prime}}=\max_{v\in V(G)} x_{v}=1$. We have the following result.
    \begin{lem}\label{lem::5}
    There exists some vertex $u^{\prime \prime} \in L_{1}(u^{\prime}) \cup L_{2}(u^{\prime})$ such that $x_{u^{\prime \prime}} \geq \frac{722}{734}$.
    \end{lem}
    \begin{proof}
Putting $u=u^{\prime}$ and $\lambda=1$ in \eqref{equ::7}, we have

$$
2 n-4<d_{G}(u^{\prime})+\sum_{v \in \overline{L_{1}}(u^{\prime})} \sum_{w \in L_{1}(u^{\prime}) \cup L_{2}(u^{\prime})} x_{w}+\frac{55 n}{9^{4}},
$$
which implies that
\begin{equation}\label{equ::8_2}
\sum_{v \in \overline{L_{1}}(u^{\prime})} \sum_{w \in L_{1}(u^{\prime}) \cup L_{2}(u^{\prime})} x_{w} \geq 2 n-4-\frac{55 n}{9^{4}}-d_{G}(u^{\prime})\geq 2 n-4-\frac{55 n}{9^{4}}-n+1 \geq \frac{722 n}{729}.
\end{equation}
As $u^\prime\in L$, by  Lemma \ref{lem::4}, we have $d_{G}(u^{\prime}) \geq \frac{725 n}{729}$. Denote by $N_{{L_{1}}(u^{\prime})}(u^{\prime})=N_{G}(u^{\prime}) \cap L_{1}(u^{\prime})$ and $d_{{L_{1}}(u^{\prime})}(u^{\prime})=|N_{{L_{1}}(u^{\prime})}(u^{\prime})|$. Then 
$$d_{{L_{1}}(u^{\prime})}(u^{\prime}) \leq |L_{1}(u^{\prime})| \leq |L| \leq 6 \times 9^{4}\leq \frac{n}{729}$$ as $n \geq 2.67 \times 9^{17}$, and it follows that 
$$
d_{{\overline{L_{1}}}(u^{\prime})}(u^{\prime})=d_{G}(u^{\prime})-d_{{L_{1}}(u^{\prime})}(u^{\prime}) \geq  \frac{724 n}{729}.
$$
Combining this with (\ref{equ::2}), we obtain
\begin{equation}\label{equ::8_3}
e(\overline{L_{1}}(u^{\prime}), L_{1}(u^{\prime}) \cup L_{2}(u^{\prime})) \leq e(\overline{L_{1}}(u^{\prime}), L)-d_{{\overline{L_{1}}}(u^{\prime})}(u^{\prime}) \leq(2 n-4)-\frac{724 n}{729} \leq \frac{734 n}{729}.
\end{equation}
According to \eqref{equ::8_2} and \eqref{equ::8_3}, there exists some $u^{\prime \prime} \in L_{1}(u^{\prime}) \cup L_{2}(u^{\prime})$ such that
$$
x_{u^{\prime \prime}} \geq \frac{\sum_{v \in \overline{L_{1}}(u^{\prime})} \sum_{w \in L_{1}(u^{\prime}) \cup L_{2}(u^{\prime})} x_{w}}{e(\overline{L_{1}}(u^{\prime}), L_{1}(u^{\prime}) \cup L_{2}(u^{\prime}))} \geq \frac{\frac{722 n}{729}}{\frac{734 n}{729}} \geq \frac{722}{734},
$$
as desired.
    \end{proof}

Recall that $x_{u^{\prime}}=1$ and $x_{u^{\prime \prime}} \geq \frac{722}{734}$ (by Lemma \ref{lem::5}). By Lemma \ref{lem::4}, we have

\begin{equation}\label{equ::9}
    d_{G}(u^{\prime}) \geq \frac{725n}{729}  ~\mbox{and}~ d_{G}(u^{\prime \prime}) \geq 0.978 n.
\end{equation}
Let $D=\{u^{\prime}, u^{\prime \prime}\}$, $R=N_{G}(u^{\prime}) \cap N_{G}(u^{\prime \prime})$, and $R_{1}=V(G) \setminus (D \cup R)$. Then
$$|R_{1}| \leq(n-d_{G}(u^{\prime}))+(n-d_{G}(u^{\prime \prime})) \leq 0.0275n.$$ 
Next we shall prove that the entries of the eigenvector $X$ with respect to the vertices in $R \cup R_{1}$ are small.
    
    \begin{lem}\label{lem::6}
    For every $u \in R \cup R_{1}$, we have $d_{R}(u) \leq 2$ and $x_{u} \leq 0.124$.
    \end{lem}

    \begin{proof}
Note that $d_{D}(u) =2$ for $u\in R$ and $d_{D}(u) \leq 1$ for $u\in R_1$. Also, for any vertex $u \in R\cup R_{1}$, we assert that $d_{R}(u) \leq 2$, since otherwise $G$ would contain a copy of $K_{3,3}$, which is impossible.  Note that $\left|R_{1}\right| \leq 0.0275n$ and $e(R_{1}) \leq 3|R_{1}|$ by (\ref{equ::2}). Then we have
$$
\begin{aligned}
\rho \sum_{u \in R_{1}} x_{u} &\leq \sum_{u \in R_{1}} d_{G}(u) =\sum_{u \in R_{1}} (d_{D}(u)+d_{R}(u)+d_{R_{1}}(u)) \leq \sum_{u \in R_{1}}(3+d_{R_{1}}(u)) \\
&\leq 3|R_{1}|+2 e(R_{1}) \leq 9|R_{1}| \leq 0.2475n,
\end{aligned}
$$
which gives that
$$\sum_{u \in R_{1}} x_{u} \leq \frac{0.2475 n}{ \rho}.$$ 
For any $u\in R\cup R_1$,
$$
\rho x_{u}=\sum_{v \in N_{G}(u)} x_{v} =\sum_{v \in N_{D}(u)} x_{v}+\sum_{v \in N_{R}(u)} x_{v}+\sum_{v \in N_{R_{1}}(u)} x_{v} \leq 4+\frac{0.2475 n}{\rho}.
$$
It follows that
$$x_{u} \leq \frac{4}{\rho}+\frac{0.2475 n}{ \rho^{2}}\leq \frac{4}{\sqrt{2 n-4}}+\frac{0.2475n}{2n-4}\leq 0.124$$
 as $\rho \geq \sqrt{2 n-4}$ and $n\geq2.67\times9^{17}$.
    \end{proof}
    
\begin{lem}\label{lem::7}
The subgraph $G[R]$ of $G$ induced by $R$ is a disjoint union of some paths and cycles. Moreover, if $G[R]$ contains a cycle then it is exactly a cycle, and if $u^{\prime}u^{\prime\prime} \in E(G)$ then $G[R]$ is a disjoint union of some paths. 
\end{lem}
\begin{proof}
Since $G$ is $K_{3,3}$-minor-free, we see that $G[R]$ is $K_{1,3}$-free, and so has maximum degree at most two. This implies that $G[R]$ must be a disjoint union of some paths and cycles. Furthermore, if $G[R]$ contains a cycle and has at least two components, then we can choose two components $H_1$ and $H_2$ from $G[R]$ such that $H_1$ is a cycle and $H_2$ is a cycle or a path. After contracting $H_1$ into a triangle, contracting $H_2$ into a  vertex $x$, and contracting $u^{\prime}xu^{\prime\prime}$ into an edge $u^{\prime}u^{\prime\prime}$, we will obtain a copy of $K_5$. This indicates that $G$ contains a  $K_{5}$-minor, which is impossible. Therefore, if $G[R]$ contains a cycle, then it is exactly a cycle. Also, if $u^{\prime}u^{\prime\prime} \in E(G)$ and $G[R]$ contains a cycle, then after contracting the cycle into a triangle, we will obtain a copy of $K_5$, which is a contradiction. Therefore, if $u^{\prime}u^{\prime\prime} \in E(G)$, then $G[R]$ is a disjoint union of some paths. 
\end{proof}

If $G$ contains $K_{2,n-2}$ as a subgraph, then we can say more about the components of the eigenvector $X$ of $\rho:=\rho(G)$.

\begin{lem}\label{lem::8}
Suppose further that $G$ contains  $K_{2,n-2}$ as a subgraph. Let  $u_1,u_2$ be the two vertices of $G$ that have degree $n-2$ in $K_{2,n-2}$. Then the following two statements hold.
\begin{enumerate}[(i)]
    \item   $x_{u_1}=x_{u_2}=1$. 
    \item For any vertex $u \in V(G)\setminus\{u_1,u_2\}$,  $\frac{2}{\rho}\leq x_u\leq \frac{2}{\rho}+ \frac{4.496}{\rho ^{2}}$.

\end{enumerate}
\end{lem}

\begin{proof}
Since $K_{2,n-2}$ is a  subgraph of $G$, we have $\rho(G)\geq \rho(K_{2,n-2})=\sqrt{2n-4}$. Let $R=N_{G}(u_1)\cap N_{G}(u_2)=V(G)\setminus\{u_1,u_2\}$.  Note that $d_{R}(u)\leq 2$ for any $u\in R$ as $G$ is $K_{3,3}$-minor-free.

(i) For any $u\in R=V(G)\setminus\{u_1,u_2\}$, we have $\rho x_u = \sum_{v\in N_{G}(u)}{x_v}\leq 4$, and it follows that  $x_u\leq \frac{4}{\rho} \leq \frac{4}{\sqrt{2n-4}}<0.124$ as $n\geq 2.67\times 9^{17}$. Since $\max_{v\in V(G)}x_v=1$, by the  symmetry of $u_1$ and $u_2$, we must have  $x_{u_1}=x_{u_2}=1$, as desired.

(ii) For any  $u \in R=V(G)\setminus\{u_1,u_2\}$, 
    \begin{equation}\label{equ::15}
         \rho x_u= x_{u_1}+x_{u_2}+\sum_{w\in N_{R}(u)}{x_w}=2+\sum_{w\in N_{R}(u)}{x_w},
    \end{equation}
which gives that  $x_u\geq \frac{2}{\rho}$.  On the other hand, since $d_R(v)\leq 2$ and $x_v\leq 0.124$ for every $v\in R$, we have 
$$ \rho \sum_{w\in N_{R}(u)}{x_w}= \sum_{w\in N_{R}(u)}{\rho x_w}= \sum_{w\in N_{R}(u)}{\sum_{y \in N_{G}(w)}{x_y}} \leq 2\times(2+0.124+0.124)=4.496.$$
Combining this with \eqref{equ::15}, we obtain $x_u\leq \frac{2}{\rho}+ \frac{4.496}{\rho ^{2}}$, as required.
\end{proof}

 \section{A structural theorem for the spectral extremal graph of $F$-free planar graphs}\label{sec::3}
 
In this section, we shall prove Theorem \ref{thm::1}, which describes a rough structure for the connected extremal graphs in $\mathrm{SPEX}_{\mathcal{P}}(n, F)$ under the condition that $F$ is a planar graph not contained in $K_{2,n-2}$ and $n\geq \max\{2.67\times9^{17},\frac{10}{9}|V(F)|\}$. The main result is as follows.

{\flushleft \it Proof of Theorem \ref{thm::1}.} Since $K_{2, n-2}$ is a connected $F$-free planar graph and $G$ is a connected extremal graph in  $\mathrm{SPEX}_{\mathcal{P}}(n, F)$, we have  $\rho \geq \rho(K_{2, n-2})=\sqrt{2 n-4}$. Combining this with $n\geq 2.67\times9^{17}$, we see that the results in Lemmas \ref{lem::2}--\ref{lem::7} hold. Let $u^{\prime}$, $u^{\prime\prime}$, $D$, $R$ and $R_1$ be defined as in Section \ref{sec::2}. Note that (ii) follows from Lemma \ref{lem::7} and (i) immediately. Thus it suffices to prove (i).

First we shall show that  $R_{1}=\emptyset$. By contradiction, suppose that  $a:=|R_{1}|>0$. Since $F$ is a planar graph with $|V(F)|\leq \frac{9n}{10}$, from (\ref{equ::9}) we obtain
\begin{equation}\label{equ::24}
\begin{aligned}
    &|R|=|N_{G}(u^{\prime}) \cap N_{G}(u^{\prime \prime})| \geq|N_{G}(u^{\prime})|+|N_{G}(u^{\prime \prime})|-n\geq 0.972 n>|V(F)|, 
 \end{aligned}
\end{equation}
and 
\begin{equation*}\label{equ::12}
|R_1|=n-|R|-2 \leq n-0.972n-2<0.028n.
\end{equation*}
Therefore,
\begin{equation}\label{equ::25}
\frac{|R|}{|R_1|}> \frac{0.972n}{0.028n} \geq 34.
\end{equation}
    Since $G[R_{1}]$ is planar, we can order the vertices of $R_1$ as $v_1,\ldots,v_a$ such that $v_1\in R_1$ satisfies $d_{R_1}(v_1)\leq 5$ and $v_i\in R_i:=R_1\setminus\{v_1,\ldots,v_{i-1}\}$ satisfies $d_{R_1\setminus\{v_1,\ldots,v_{i-1}\}}(v_i)\leq 5$ for $i\in\{2,\ldots,a\}$ (if $a\geq 2$).  By Lemma \ref{lem::6}, for every $i\in\{1,\ldots,a\}$, 
\begin{equation}\label{equ::26}
  \sum_{w \in N_{D \cup R \cup R_{i}}(v_{i})} x_{w} \leq 1+\sum_{w \in N_{R}(v_{i})} x_{w}+\sum_{w \in N_{R_{i}}(v_{i})} x_{w} \leq 1+2\times 0.124+5\times 0.124=1.868.  
\end{equation} 
Also note that $\cup_{i=1}^{a}\{w v_{i} \mid w \in N_{D \cup R \cup R_{i}}(v_{i})\}$ is exactly the set of edges in $G$ that is incident with some vertex of $R_{1}$. We consider the following  two cases. 

{\flushleft {\it Case 1.} $u^{\prime}u^{\prime\prime} \in E(G)$.}

In this case, by Lemma \ref{lem::7},  $G[R]$ is a disjoint union of some paths. Let $G^{\prime}:=G-\sum_{i=1}^{a} \sum_{w \in N_{D \cup R \cup R_{i}}(v_{i})} {wv_{i}}+\sum_{i=1}^{a}{(v_iu^{\prime}+v_iu^{\prime \prime})}$. Clearly, $G^{\prime}$ is a connected subgraph of $K_2+P_{n-2}$, and so is planar. Recall that $x_{u^\prime}=1$ and $x_{u^{\prime\prime}}\geq \frac{722}{734}>0.983$. According to \eqref{equ::26}, we obtain
$$
\rho(G^{\prime})-\rho \geq \frac{2}{X^{T} X} \sum_{i=1}^{a} x_{v_{i}}\left(x_{u^{\prime}}+x_{u^{\prime \prime}}-\sum_{w \in N_{D \cup R \cup R_{i}}(v_{i})} x_{w}\right)>0,
$$
and hence $\rho(G^{\prime})>\rho$. Now we shall show that $G^{\prime}$ is $F$-free. Suppose to the contrary that $G^{\prime}$ contains $F$ as a subgraph. Clearly, $V(F) \cap R_{1}\neq\emptyset$. As $|R|>|V(F)|$ by (\ref{equ::24}), we have
\begin{equation}\label{equ::27}
|R \setminus  V(F)|=|R|-|R \cap V(F)|>|V(F)|-|V(F) \cap R| \geq|V(F) \cap R_{1}|.
\end{equation}
Suppose that $V(F) \cap R_{1}=\{v_{i_1}, \ldots, v_{i_b}\}$. Take $w_{1}, \ldots, w_{b}\in R \setminus  V(F)$. Then we see that   $N_{G^{\prime}}(v_{i_j})=D \subseteq N_{G^{\prime}}(w_{j})$ for all $j \in\{1, \ldots, b\}$, and $G$ has already contained a copy of $F$, which is impossible. Thus we may conclude that $G'$ is an $n$-vertex connected $F$-free planar graph with  $\rho(G^{\prime})>\rho$. However, this is impossible by the maximality of $\rho$. 

{\flushleft {\it Case 2.} $u^{\prime}u^{\prime\prime} \notin E(G)$.}

In this case, by Lemma \ref{lem::7}, $G[R]$ is a disjoint union of some paths and cycles. If $G[R]$ is a disjoint union of some paths, let  $G^{\prime}=G-\sum_{i=1}^{a} \sum_{w \in N_{D \cup R \cup R_{i}}(v_{i})} {wv_{i}}+\sum_{i=1}^{a}{(v_iu^{\prime}+v_iu^{\prime \prime})}$. As in Case 1, we see that $G^\prime$ is a connected $F$-free planar graph with $\rho(G')>\rho$, which is a contradiction. Thus $G[R]$ contains a cycle, and  so must be a cycle according to Lemma \ref{lem::7}. We have the following claim.

\begin{claim}\label{claim::3.1}
Let $v$ be a vertex of $R_1$ with  $d_{R}(v)=2$. Then the two neighbors of $v$ in $R$ are adjacent.
\end{claim}
\begin{proof}
Let $N_{R}(v)=\{u_1,u_2\}$.  Suppose to the contrary that $u_1$ and $u_2$ are not adjacent. Then $u_1$ and $u_2$ are two non-consecutive vertices on the cycle $G[R]$. Let $P$ and $P^{\prime}$ be the two internally vertex-disjoint paths of length at least $2$ between $u_1$ and $u_2$ along this cycle. After contracting $u_{1}vu_2$ into an edge $u_{1}u_{2}$, contracting the resulting cycle $P\cup u_{1}u_{2}$ into a triangle, picking an internal vertex $\overline{u}$ from $P^{\prime}$, and contracting   $u^{\prime}\overline{u}u^{\prime\prime}$ into an edge $u^{\prime}u^{\prime\prime}$, we will obtain a copy of $K_5$. This implies that $G$ has a $K_{5}$-minor, contrary to our assumption.
\end{proof}

{\flushleft {\it Subcase 2.1.} There exists some vertex $v_r\in R_1$ adjacent to  $u^{\prime}$ or $u^{\prime\prime}$.}

Since $v_r$ is adjacent to $u^{\prime}$ or $u^{\prime\prime}$, we have 
$\rho x_{v_r}\geq \min\{x_{u^{\prime}},x_{u^{\prime\prime}}\}\geq \frac{722}{734}>0.983$, and it follows from Lemma \ref{lem::1} that
\begin{equation}\label{equ::28}
 x_{v_r} > \frac{0.983}{\rho} \geq \frac{0.983}{2+\sqrt{2n-6}}.
\end{equation}
On the other hand, by Lemma \ref{lem::6}, (\ref{equ::25}) and Claim \ref{claim::3.1}, we can easily find two adjacent vertices $u_1,u_2\in R$ such that $N_G(u_1)\cap R_1=N_G(u_2)\cap R_1=\emptyset$. Then, again by Lemma \ref{lem::6}, we obtain 
\begin{equation*}\label{equ::17}
    \rho x_{u_i}= x_{u^{\prime}} + x_{u^{\prime\prime}} +\sum_{w \in N_{R}(u_i)}x_w \leq 1+1+2 \times 0.124 = 2.248~~\mbox{for $i \in \{1,2\}$}.
\end{equation*}
Combining this with $\rho \geq\sqrt{2 n-4}$ yields that
\begin{equation}\label{equ::29}
     x_{u_i} \leq \frac{2.248}{\rho} \leq \frac{2.248}{\sqrt{2n-4}}~~\mbox{for $i \in \{1,2\}$}.
\end{equation}
Let $G^{\prime}=G-u_{1}u_{2}-\sum_{i=1}^{a} \sum_{w \in N_{D \cup R \cup R_{i}}(v_{i})} {wv_{i}}+\sum_{i=1}^{a}(v_iu^{\prime}+v_iu^{\prime \prime})$. Clearly, $G^{\prime}$ is a connected subgraph of $2K_1+P_{n-2}$, and so is planar. Recall that $x_{u^\prime}=1$ and $x_{u^{\prime\prime}}\geq \frac{722}{734}>0.983$.  Combining (\ref{equ::26}), (\ref{equ::28}) and (\ref{equ::29}), we get

$$
\begin{aligned}
    \rho(G^{\prime})-\rho & \geq \frac{2}{X^{T} X} \left(\sum_{i=1}^{a} x_{v_{i}}\left( x_{u^{\prime}}+x_{u^{\prime \prime}}-\sum_{w \in N_{D \cup R \cup R_{i}}(v_{i})} x_{w}\right)-x_{u_1}x_{u_2}\right) \\
    &\geq \frac{2}{X^{T} X} \left(\sum_{i=1}^{a} x_{v_{i}}(1+0.983-1.868)-\frac{2.248^{2}}{2n-4}\right)\\
    &\geq \frac{2}{X^{T} X} \left( x_{v_{r}}\times 0.115-\frac{2.248^{2}}{2n-4}\right)\\
    & \geq \frac{2}{X^{T} X} \left(\frac{0.983\times 0.115}{2+\sqrt{2n-6}}-\frac{2.248^{2}}{2n-4}\right)\\
    & >0
\end{aligned}
$$
as $\rho \geq \sqrt{2 n-4}$ and $n\geq 2.67 \times 9^{17}$. Therefore, $\rho(G')>\rho$. Furthermore, as in Case 1, we see that $G^{\prime}$ is $F$-free, which is impossible by the maximality of $\rho=\rho(G)$.

{\flushleft {\it Subcase 2.2.} $N_{R_1}(u^{\prime})=N_{R_1}(u^{\prime\prime})=\emptyset$.}

In this situation, by Lemma \ref{lem::6}, we obtain
\begin{equation}\label{equ::30}
    \sum_{w \in N_{D \cup R \cup R_{i}}(v_{i})} x_{w} =\sum_{w \in N_{R}(v_{i})} x_{w}+\sum_{w \in N_{R_{i}}(v_{i})} x_{w} <2\times 0.124+5\times 0.124=0.868.  
\end{equation}
Let $G^{\prime}=G-\sum_{i=1}^{a} \sum_{w \in N_{D \cup R \cup R_{i}}(v_{i})} {wv_{i}}+\sum_{i=1}^{a}{v_{i}u^{\prime}}$. Clearly, $G^{\prime}$ is a connected planar graph. Recall that $x_{u^\prime}=1$. According to (\ref{equ::30}), we have
$$
\rho(G^{\prime})-\rho \geq \frac{2}{X^{T} X} \sum_{i=1}^{a} x_{v_{i}}\left(x_{u^{\prime}}-\sum_{w \in N_{D \cup R \cup R_{i}}(v_{i})} x_{w}\right)>0,
$$
and hence $\rho(G^\prime)>\rho$. As in Case 1, we see that $G^{\prime}$ is $F$-free, which is  impossible by the maximality of $\rho=\rho(G)$.

According to the above discussions, we obtain $R_1=\emptyset$, and so $R=N_G(u^\prime)\cap N_G(u^{\prime\prime})=V(G)\setminus\{u^\prime,u^{\prime\prime}\}$. Moreover, by the symmetry of $u^\prime$ and $u^{\prime\prime}$, we can deduce that $x_{u^\prime}=x_{u^{\prime\prime}}=1$. In particular,  $G$ contains a copy of $K_{2,n-2}$. This proves (i).\qed

\section{Proof of Theorems \ref{thm::2}--\ref{thm::4}}\label{sec::4}

In this section, by using the structure theorem for the spectral extremal graph of $F$-free planar  graphs obtained in Section \ref{sec::3}, we shall give the proof of Theorems \ref{thm::2}--\ref{thm::4}. More specifically, we characterize the unique extremal graph in $\mathrm{SPEX}_{\mathcal{P}}(n, W_k)$, $\mathrm{SPEX}_{\mathcal{P}}(n, F_k)$ and $\mathrm{SPEX}_{\mathcal{P}}(n, (k+1)K_2)$, respectively, provided that $n$ is  sufficiently large relative to $k$. To achieve this goal, we first need to discuss the change of the spectral radius of a special kind of planar graphs under specified graph transformation. 

Suppose that $H$ is a disjoint union of at least two paths. Let $P_{s_1}$ and $P_{s_2}$ be any two components of $H$ with $s_1\geq s_2$. Then the $(s_{1}, s_{2})$-\textit{transformation} $H^*$ of $H$ is defined by
$$
H^{*}:=\begin{cases}P_{s_{1}+1} \cup P_{s_{2}-1} \cup H_0, & \mbox{if } s_{2} \geq 2, \\ P_{s_{1}+s_{2}} \cup H_0, & \mbox{if } s_{2}=1,\end{cases}
$$
where $H_0$ is the union of the components of $H$ other than $P_{s_1}$ and $P_{s_2}$.
Clearly, both $K_2+H$ and $K_{2}+H^{*}$ are connected planar graphs. The following lemma indicates that $\rho(K_{2}+H)<\rho(K_{2}+H^*)$ under certain conditions.

\begin{lem}\label{lem::9}
Let $H$ be a graph on $n-2$ vertices of which all components are paths. Suppose that $P_{s_1}$ and $P_{s_2}$ are two components of $H$ with $s_1\geq s_2$. If $n \geq \max\{2.67\times 9^{17},10.2 \times 2^{s_{2}}+2\}$, then $\rho(K_2+H)<\rho(K_2+H^*)$, where $H^*$ is  the $(s_{1}, s_{2})$-\textit{transformation} of $H$.
\end{lem}

\begin{proof}
Note that both $K_2+H$ and $K_{2}+H^{*}$ are connected planar graphs of order $n$ containing $K_{2,n-2}$ as a subgraph. Let $V(K_2)=\{u_1,u_2\}$, $\rho=\rho(K_2+H)$ and $\rho^*=\rho(K_2+H^*)$. Clearly, $\rho\geq \rho(K_{2,n-2})=\sqrt{2n-4}$. Let $X=(x_v: v\in V(K_2+H))^T$ be the positive eigenvector of $\rho$ with $\max_{v\in V(K_2+H)}x_v=1$. By Lemma \ref{lem::8} (i), we have $x_{u_1}=x_{u_2}=1$. Suppose that $P_{s_1}=v_{1} v_{2} \cdots v_{s_{1}}$ and $P_{s_2}=w_{1} w_{2} \cdots w_{s_{2}}$.  If $s_{2}=1$, then $H$ is a proper subgraph of  $H^{*}$. This implies that $K_{2}+H$ is a proper subgraph of $K_{2}+H^{*}$, and hence  $\rho<\rho^{*}$, as required. Now suppose $s_{2}=2$. If $x_{v_{1}} \leq x_{w_{1}}$, then after deleting the edge $v_{1} v_{2}$ and adding the edge $v_{2} w_{1}$ in $H$ we will obtain the graph $H^*$. Thus we have
$$
\rho^{*}-\rho \geq \frac{X^{T}(A(K_{2}+H^{*})-A(K_{2}+H)) X}{X^{T} X}=\frac{2}{X^{T} X}(x_{w_{1}}-x_{v_{1}}) x_{v_{2}} \geq 0.
$$
Since $X$ is an eigenvector of $\rho$, we obtain $\rho x_{v_{1}}=2+x_{v_{2}}$. If $\rho=\rho^{*}$, then $X$ is also an eigenvector of $\rho^{*}$, and hence $\rho x_{v_{1}}=\rho^{*} x_{v_{1}}=2$. However, this is impossible because $\rho x_{v_{1}}=2+x_{v_{2}}$ and $x_{v_2}>0$. Therefore, $\rho<\rho^{*}$. If $x_{v_{1}}>x_{w_{1}}$, then after deleting the edge $w_1w_2$ and adding the edge $v_1w_2$ in $H$ we will obtain the graph $H^*$ again.  By using a similar analysis, we also can prove that $\rho<\rho^{*}$. Thus it suffices to consider the case that $s_{2} \geq 3$. Before going further, we need the following result.

\begin{claim}\label{claim::4.0}
Let $i$ be a positive integer, and let  $A_{i}=\left[\frac{2}{\rho}-\frac{4.496 \times 2^{i}}{\rho^{2}}, \frac{2}{\rho}+\frac{4.496 \times 2^{i}}{\rho^{2}}\right]$ and $B_{i}=\left[-\frac{4.496 \times 2^{i}}{\rho^{2}}, \frac{4.496 \times 2^{i}}{\rho^{2}}\right]$. Then the following statements hold:
\begin{enumerate}[(a)]
    \item $\rho^{i}(x_{v_{i+1}}-x_{v_{i}}) \in A_{i}$ for $1\leq i\leq \lfloor\frac{s_{1}-1}{2}\rfloor$;
    \item $\rho^{i}(x_{w_{i+1}}-x_{w_{i}}) \in A_{i}$ for $1\leq i\leq \lfloor\frac{s_{2}-1}{2}\rfloor$;
    \item $\rho^{i}(x_{v_{i}}-x_{w_{i}}) \in B_{i}$ for $1\leq i\leq \lfloor\frac{s_{2}}{2}\rfloor$.
\end{enumerate}
\end{claim}
\begin{proof}
    
(a) It suffices to prove that for every $i\in \{1,\ldots, \lfloor\frac{s_{1}-1}{2}\rfloor\}$,
$$
\rho^{i}(x_{v_{j+1}}-x_{v_{j}}) \in \begin{cases}A_{i}, & \mbox{if } j=i, \\ B_{i}, & \mbox{if } i+1 \leq j \leq s_{1}-i-1.\end{cases}
$$ 
We shall proceed the proof by induction on $i$. Clearly,

\begin{equation}\label{equ::16}
    \rho x_{v_{j}}=x_{u_1}+x_{u_2}+\sum_{v \in N_{R}(v_{j})} x_{v}= \begin{cases}2+x_{v_{2}}, & \mbox{if } j=1, \\ 2+x_{v_{j-1}}+x_{v_{j+1}}, & \mbox {if } 2 \leq j \leq s_{1}-1.\end{cases}
\end{equation}
By  \eqref{equ::16} and Lemma \ref{lem::8} (ii), we have
\begin{equation*}
    \rho(x_{v_{j+1}}-x_{v_{j}})= \begin{cases}
    x_{v_{1}}+x_{v_{3}}-x_{v_{2}} \in A_{1}, & \mbox{if } j=1,\\ (x_{v_{j}}-x_{v_{j-1}})+(x_{v_{j+2}}-x_{v_{j+1}}) \in B_{1}, & \mbox{if } 2 \leq j \leq s_{1}-2.
    \end{cases}
\end{equation*}
Thus the result holds for $i=1$. Now suppose $2 \leq i \leq\lfloor\frac{s_{1}-1}{2}\rfloor$, and assume that the result holds for $i-1$, that is,
\begin{equation}\label{equ::31}
\rho^{i-1}(x_{v_{l+1}}-x_{v_{l}})\in \begin{cases}A_{i-1}, & \mbox{if } l=i-1, \\ B_{i-1}, & \mbox{if } i \leq l \leq s_{1}-i.\end{cases}
\end{equation} 
Note that, for each $j\in\{i,\ldots, s_{1}-i-1\}$,  $\rho(x_{v_{j+1}}-x_{v_{j}})=(x_{v_{j}}-x_{v_{j-1}})+(x_{v_{j+2}}-x_{v_{j+1}})$, and it follows that
\begin{equation}\label{equ::18}
    \rho^{i}(x_{v_{j+1}}-x_{v_{j}})=\rho^{i-1}(x_{v_{j}}-x_{v_{j-1}})+\rho^{i-1}(x_{v_{j+2}}-x_{v_{j+1}}) .
\end{equation}
If $j=i$, then \eqref{equ::31} implies that  $\rho^{i-1}(x_{v_{j}}-x_{v_{j-1}}) \in A_{i-1}$ and $\rho^{i-1}(x_{v_{j+2}}-x_{v_{j+1}}) \in B_{i-1}$, and hence $\rho^{i}(x_{v_{j+1}}-x_{v_{j}})\in A_i$ by \eqref{equ::18}, as desired. If $i+1 \leq j \leq s_{1}-i-1$, then \eqref{equ::31} implies that  $\rho^{i-1}(x_{v_{j}}-x_{v_{j-1}}) \in B_{i-1}$ and $\rho^{i-1}(x_{v_{j+2}}-x_{v_{j+1}}) \in B_{i-1}$, and hence $\rho^{i}(x_{v_{j+1}}-x_{v_{j}})\in B_i$ by \eqref{equ::18}. Thus the result follows.

(b) The proof is similar to (a), and   we omit it here.

(c) It suffices to prove that $\rho^{i}(x_{v_{j}}-x_{w_{j}}) \in B_{i}$ for all $i \in\{1, \ldots,\lfloor\frac{s_{2}}{2}\rfloor\}$ and  $j \in\{i,\ldots,s_{2}-i\}$. We shall proceed the proof by induction on $i$. Clearly,
$$
\rho x_{w_{j}}=x_{u_1}+x_{u_2}+\sum_{w \in N_{R}(w_{j})} x_{w}= \begin{cases}2+x_{w_{2}}, & \text{if } j=1, \\ 2+x_{w_{j-1}}+x_{w_{j+1}}, & \text{if } 2 \leq j \leq s_{2}-1 .\end{cases}
$$
Combining this with (\ref{equ::16}) and Lemma \ref{lem::8} (ii), we obtain
$$
\begin{aligned}
    \rho(x_{v_{j}}-x_{w_{j}})= \begin{cases}x_{v_{2}}-x_{w_{2}} \in B_{1}, & \text{if } j=1, \\ (x_{v_{j+1}}-x_{w_{j+1}})+(x_{v_{j-1}}-x_{w_{j-1}}) \in B_{1}, & \text{if } 2 \leq j \leq s_{2}-1.\end{cases}
\end{aligned}
$$
Thus the result holds for $i=1$. Now suppose $2 \leq i \leq\lfloor\frac{s_{2}}{2}\rfloor$, and assume that the result holds for $i-1$, that is, $\rho^{i-1}(x_{v_l}-x_{w_l})\in B_{i-1}$ whenever $l\in\{i-1,\ldots,s_2-i+1\}$. For each $j\in \{i,\ldots,s_{2}-i\}$, we see that $\rho(x_{v_{j}}-x_{w_{j}})=(x_{v_{j-1}}-x_{w_{j-1}})+(x_{v_{j+1}}-x_{w_{j+1}})$, and hence
\begin{equation}\label{equ::19}
    \rho^{i}(x_{v_{j}}-x_{w_{j}})=\rho^{i-1}(x_{v_{j-1}}-x_{w_{j-1}})+\rho^{i-1}(x_{v_{j+1}}-x_{w_{j+1}}) .
\end{equation}
By the induction hypothesis, $\rho^{i-1}(x_{v_{j-1}}-x_{w_{j-1}}) \in B_{i-1}$ and $\rho^{i-1}(x_{v_{j+1}}-x_{w_{j+1}}) \in B_{i-1}$. Then it follows from  (\ref{equ::19}) that $\rho^{i}(x_{v_{j}}-x_{w_{j}})\in B_{i}$, as desired. 
\end{proof} 
Since $n \geq 10.2 \times 2^{s_{2}}+2$, we have $\rho \geq \sqrt{2 n-4}>4.496 \times 2^{s_{2}/2}$. 
Combining this with Claim \ref{claim::4.0}, we obtain
\begin{equation}\label{equ::20}
    x_{v_{i+1}}-x_{v_{i}} \geq \frac{2}{\rho^{i+1}}-\frac{4.496 \times 2^{i}}{\rho^{i+2}}>0
\end{equation}
whenever  $i \leq \min \{\frac{s_{2}}{2},\lfloor\frac{s_{1}-1}{2}\rfloor\}$
 and
\begin{equation}\label{equ::21}
    x_{v_{i+1}}-x_{w_{i}}=(x_{v_{i+1}}-x_{v_{i}})+(x_{v_{i}}-x_{w_{i}})\geq\left(\frac{2}{\rho^{i+1}}-\frac{4.496 \times 2^{i}}{\rho^{i+2}}\right)-\frac{4.496 \times 2^{i}}{\rho^{i+2}}>0
\end{equation}
whenever $i \leq \min \{\lfloor\frac{s_{2}}{2}\rfloor,\lfloor\frac{s_{1}-1}{2}\rfloor\}$. Similarly,
we also can deduce that
\begin{equation}\label{equ::22}
    x_{w_{i+1}}>x_{w_{i}} \mbox{ and } x_{w_{i+1}}>x_{v_{i}}
\end{equation}
whenever  $i\leq\lfloor\frac{s_{2}-1}{2}\rfloor$.
Recall that $s_{2} \geq 3$. Let $t_{1}$ and $t_{2}$ be two positive integers such that $t_{1}+t_{2}=s_{2}-1$. Then after deleting the edges $v_{t_{1}} v_{t_{1}+1},w_{t_{2}} w_{t_{2}+1}$  and adding the edges $v_{t_{1}} w_{t_{2}}, v_{t_{1}+1} w_{t_{2}+1}$ in $H$, we will obtain the graph $H^*$. Thus we have
\begin{equation}\label{equ::23}
    \rho^{*}-\rho \geq \frac{X^{T}(A(K_{2}+H^{*})-A(K_{2}+H)) X}{X^{T} X}\geq \frac{2}{X^{T} X}(x_{v_{t_{1}+1}}-x_{w_{t_{2}}})(x_{w_{t_{2}+1}}-x_{v_{t_{1}}}).
\end{equation}
If $s_{2}$ is odd, we take $t_1=t_2=\frac{s_{2}-1}{2}$. Then it follows from (\ref{equ::21}) and (\ref{equ::22}) that $x_{v_{t_{1}+1}}>x_{w_{t_{2}}}$ and $x_{w_{t_{2}+1}}>x_{v_{t_{1}}}$. Therefore, $\rho<\rho^{*}$ by (\ref{equ::23}), as desired. If $s_{2}$ is even, we only consider the case that $x_{w_{s_{2} / 2}} \geq x_{v_{s_{2} / 2}}$, since the proof for the case that $x_{w_{s_{2} / 2}}<x_{v_{s_{2} / 2}}$ is similar. Take $t_{1}=\frac{s_{2}}{2}$ and  $t_{2}=\frac{s_{2}}{2}-1$. Then we have  $x_{w_{t_{2}+1}} \geq x_{v_{t_1}}$ as $x_{w_{s_{2} / 2}} \geq x_{v_{s_{2} / 2}}$. If $s_{1}=s_{2}$, then $s_{1}$ is even, and hence $x_{v_{s_{1}/ 2+1}}=x_{v_{s_{1} / 2}}$ by symmetry, that is, $x_{v_{t_{1}+1}}=x_{v_{t_{1}}}$. If $s_{1} \geq s_{2}+1$, then  $x_{v_{t_{1}+1}}>x_{v_{t_{1}}}$ by (\ref{equ::20}). In both cases, we have $x_{v_{t_{1}+1}} \geq x_{v_{t_{1}}}$. Furthermore, from (\ref{equ::21}) we get $x_{v_{t_{1}}}>x_{w_{t_{2}}}$, and so $x_{v_{t_{1}+1}}>x_{w_{t_{2}}}$. Thus $\rho \leq \rho^{*}$ by (\ref{equ::23}). If $\rho=\rho^{*}$, then $X$ is also an eigenvector of $\rho^{*}$, and so $\rho x_{v_{t_{1}}}=\rho^{*} x_{v_{t_{1}}}=$ $2+x_{v_{t_{1}-1}}+x_{w_{t_{2}}}$. On the other hand, since $X$ is an eigenvector of $\rho$, we have $\rho x_{v_{t_{1}}}=2+x_{v_{t_{1}-1}}+x_{v_{t_{1}+1}}$. Hence,  $x_{v_{t_{1}+1}}=x_{w_{t_{2}}}$, which is a contradiction. Therefore, we conclude that $\rho<\rho^{*}$, and the result follows.

We complete the proof.
\end{proof}

\subsection{The spectral extremal graph for $W_k$-free planar graphs}

In this subsection, we shall prove Theorem \ref{thm::2}, which determines the unique extremal graph in $\mathrm{SPEX}_{\mathcal{P}}(n, W_k)$ for every integer  $k\geq 3$, provided that $n$ is sufficiently large relative to $k$. 

{\flushleft \it Proof of Theorem \ref{thm::2}.} Let $k\geq 3$ and  $n\geq \max\{2.67\times9^{17},10.2\times 2^{k-4}+2\}$. It is easy to verify that $W_{n,k}$ is planar and $W_k$-free for each $k\geq 3$. Suppose that $G$ is an extremal graph in $\mathrm{SPEX}_{\mathcal{P}}(n,W_k)$. First we claim that $G$ is connected, since otherwise we can add some new edges into $G$ such that the resulting graph is still $W_k$-free and planar, but has larger spectral radius than $G$, which is impossible. Note that $W_k$ is a planar graph that is not contained in $K_{2,n-2}$. Since $n\geq \max\{2.67\times9^{17},10.2\times 2^{k-4}+2\}>\frac{10}{9}k=\frac{10}{9}|V(W_k)|$, by Theorem \ref{thm::1} (i), there exist two vertices  $u^\prime,u^{\prime\prime}\in V(G)$ such that $R=N_G(u^\prime)\cap N_G(u^{\prime\prime})=V(G)\setminus\{u^{\prime},u^{\prime\prime}\}$ and $x_{u^{\prime}}=x_{u^{\prime\prime}}=1$. In particular,  $G$ contains a copy of $K_{2,n-2}$.

If $k=3$, then $W_k$ is a triangle. Since adding an arbitrary edge into $K_{2,n-2}$ would result in a triangle, we assert that $G\cong K_{2,n-2}=W_{n,3}$, as desired. If $k=4$, then we assert that $u^{\prime}u^{\prime\prime} \notin E(G)$. Indeed, if  $u^{\prime}u^{\prime\prime} \in E(G)$, then $G[R]$ is an empty graph, since otherwise $G$ would contain a copy of $W_4$. Hence, $G\cong K_2 +(n-2)K_1$. Let $R=\{u_1,u_2,\dots, u_{n-2}\}$ and $G^{\prime}=G-u^{\prime}u^{\prime\prime}+u_{1}u_{n-2}+\sum_{i=1}^{n-3}{u_{i}u_{i+1}}$. Clearly, $G^{\prime}\cong 2K_1 +C_{n-2}$, which is also planar and  $W_4$-free because $n\geq 2.67\times 9^{17}$. By Lemma \ref{lem::1} and Lemma \ref{lem::8} (ii), we have
    $$
    \begin{aligned}
        \rho(G^{\prime})-\rho &\geq \frac{2}{X^{T} X} \left(x_{u_1}x_{u_{n-2}}+\sum_{i=1}^{n-3}{x_{u_i}x_{u_{i+1}}}-x_{u^{\prime}}x_{u^{\prime\prime}}\right) \\
        & \geq \frac{2}{X^{T} X} \left( \left(\frac{2}{\rho}\right)^{2}\cdot (n-2)-1\right)\\
        &  \geq \frac{2}{X^{T} X} \left( \left(\frac{2}{2+\sqrt{2n-6}}\right)^2\cdot(n-2)-1\right)\\
        & >0,
    \end{aligned}
    $$
and hence $\rho(G^{\prime})>\rho$, contrary to the maximality of $\rho=\rho(G)$.  Therefore, $u^{\prime}u^{\prime\prime} \notin E(G)$. By Theorem \ref{thm::1} (ii), $G[R]$ is  a disjoint union of some paths and cycles. If $G[R]$ is a disjoint union of some paths, then $G$ is a subgraph of $2K_1+P_{n-2}$, and so must be a proper subgraph of $2K_1+C_{n-2}$, which is impossible by the maximality of $\rho=\rho(G)$. If $G[R]$ contains a cycle, again by Theorem \ref{thm::1} (ii), we obtain $G[R]\cong C_{n-2}$, and hence $G\cong 2K_1+C_{n-2}=W_{n,4}$, as required. Now suppose that $k\geq 5$. We have the following two claims.
\begin{claim}\label{claim::4.1}
     The vertices $u^\prime$ and $u^{\prime\prime}$ are adjacent, and $G[R]$ is $P_{k-2}$-free.
\end{claim}
\begin{proof}
First assume that $k=5$. If $G[R]$ contains a copy of $P_3=u_1u_2u_3$, then $u^{\prime}$, $u^{\prime\prime}$, $u_1$, $u_2$ and $u_3$ would generate a copy of $W_5$, a contradiction. Thus $G[R]$ is $P_3$-free, and so must be a disjoint union of some independent edges and isolated vertices. If $u^{\prime}u^{\prime\prime}\notin E(G)$, let $G^\prime=G+u^{\prime}u^{\prime\prime}$. Then it is easy to see that $G'$ is also planar and $W_5$-free. However, $\rho(G^\prime)>\rho$, which contradicts the maximality of $\rho=\rho(G)$. Therefore, $u^{\prime}u^{\prime\prime}\in E(G)$.

Now assume that $k\geq 6$. Suppose to the contrary that $u^{\prime}u^{\prime\prime}\notin E(G)$. By Theorem \ref{thm::1} (ii), $G[R]$ is a union of some paths and cycles. Combining this with  $u^{\prime}u^{\prime\prime}\notin E(G)$, we see that $G$ must be a subgraph of $2K_1+C_{n-2}$. Also note that $2K_1+C_{n-2}$ is planar and $W_k$-free because $k\geq 6$ and $n\geq \max\{2.67\times 9^{17}, 10.2\times 2^{k-4}+2\}>k+1$. Therefore, we conclude that $G\cong 2K_1+C_{n-2}$ by the maximality of $\rho=\rho(G)$.  Let $G[R]=u_{1}u_{2}u_{3}\cdots u_{n-2}u_{1}$, and let 
$$ 
G^{\prime}:= 
          \begin{cases}G+u^{\prime}u^{\prime\prime}-u_{1}u_{n-2}-\sum_{i=1}^{\frac{n-2}{3}-1}{u_{3i}u_{3i+1}}, & \mbox{if $n-2\equiv 0 \pmod 3$},\\
     G+u^{\prime}u^{\prime\prime}-u_{1}u_{n-2}-\sum_{i=1}^{\lfloor\frac{n-2}{3}\rfloor}{u_{3i}u_{3i+1}}, & \mbox{otherwise}.
          \end{cases} 
$$
Clearly, $G^{\prime}$ is $W_k$-free because $k\geq 6$ and $G^{\prime}[R]$ is $P_4$-free. Moreover, $G^{\prime}$ is a planar graph because it is a subgraph of  $K_2+P_{n-2}$.   By Lemma \ref{lem::8} (ii),
    if $n-2\equiv 0 \pmod 3$, then 
     $$
    \begin{aligned}
        \rho(G^{\prime})-\rho &\geq \frac{2}{X^{T} X} \left(x_{u^{\prime}}x_{u^{\prime\prime}}-x_{u_1}x_{u_{n-2}}-\sum_{i=1}^{\frac{n-2}{3}-1}{x_{u_{3i}}x_{u_{3i+1}}}\right) \\
        & \geq \frac{2}{X^{T} X} \left( 1-\left(\frac{n-2}{3}-1+1\right)\cdot \left(\frac{2}{\rho}+\frac{4.496}{\rho^{2}}\right)^{2}\right)\\
        &  \geq \frac{2}{X^{T} X} \left( 1-\frac{n-2}{3}\cdot \left(\frac{2}{\sqrt{2n-4}}+\frac{4.496}{2n-4}\right)^{2}\right)\\
        & >0,
    \end{aligned}
    $$
and if $n-2\not\equiv 0 \pmod 3$, then
    $$
    \begin{aligned}
        \rho(G^{\prime})-\rho &\geq \frac{2}{X^{T} X} \left(x_{u^{\prime}}x_{u^{\prime\prime}}-x_{u_1}x_{u_{n-2}}-\sum_{i=1}^{\left\lfloor\frac{n-2}{3}\right\rfloor}{x_{u_{3i}}x_{u_{3i+1}}}\right) \\
        & \geq \frac{2}{X^{T} X} \left( 1-\left(\left\lfloor\frac{n-2}{3}\right\rfloor+1\right)\cdot\left(\frac{2}{\rho}+\frac{4.496}{\rho^{2}}\right)^{2}\right)\\
        &  \geq \frac{2}{X^{T} X} \left( 1-\frac{n+1}{3}\cdot\left(\frac{2}{\sqrt{2n-4}}+\frac{4.496}{2n-4}\right)^{2}\right)\\
        & >0.
    \end{aligned}
    $$
Thus we have $\rho(G^\prime)>\rho$, contrary to the maximality of $\rho=\rho(G)$. Therefore, $u^{\prime}u^{\prime\prime}\in E(G)$, and it follows that $G[R]$ is $P_{k-2}$-free because $G$ is $W_k$-free.
\end{proof}

According to Theorem \ref{thm::1} (ii)  and Claim \ref{claim::4.1}, $G[R]$ is a disjoint union of some  paths with length at most $k-3$. Moreover, we see that $G[R]$ contains at least two components as $n\geq \max\{2.67\times9^{17},10.2\times 2^{k-4}+2\}$. Let $H^*=\lfloor\frac{n-2}{k-3}\rfloor P_{k-3} \cup P_{n-2-(k-3)\cdot\lfloor\frac{n-2}{k-3}\rfloor}$. Clearly, $K_2+H^*=W_{n,k}$. If $G[R]\cong H^*$, then $G\cong W_{n,k}$, and we are done. If $G[R]\ncong H^*$, then we can obtain $H^*$ by applying a series of $(s_1,s_2)$-transformations to $G[R]$,  where $s_2\leq s_1\leq k-4$. As $n\geq 10.2\times 2^{k-4}+2\geq 10.2\times 2^{s_2}+2$, by Lemma \ref{lem::9}, we conclude that $\rho<\rho(K_2+H^*)=\rho(W_{n,k})$, which is impossible by the maximality of $\rho=\rho(G)$.

This completes the proof.\qed

\subsection{The spectral extremal graph for $F_k$-free planar graphs}

In this subsection, we shall prove Theorem \ref{thm::3}, which determines the unique graph
in $\mathrm{SPEX}_{\mathcal{P}}(n, F_k)$ for every integer  $k\geq 1$, provided that $n$ is sufficiently large relative to $k$.

{\flushleft \it Proof of Theorem \ref{thm::3}.}  Suppose that  $n\geq \max\{2.67\times9^{17},10.2\times2^{2k-4}+2\}$. It is easy to verify that $F_{n,k}$ is planar and $F_{k}$-free for all $k\geq 1$. Suppose that $G$ is an extremal graph in $\mathrm{SPEX}_{\mathcal{P}}(n,F_k)$. First we claim that $G$ is connected, since otherwise we can add some new edges into $G$ such that the resulting graph is still $F_k$-free and planar, but has larger spectral radius than $G$, which is impossible. Note that $F_k$ is a planar graph that is not contained in $K_{2,n-2}$. Since $n\geq \max\{2.67\times9^{17},10.2\times2^{2k-4}+2\}>\frac{10}{9}(2k+1)=\frac{10}{9}|V(F_k)|$, by Theorem \ref{thm::1} (i), there exist two vertices  $u^\prime,u^{\prime\prime}\in V(G)$ such that $R=N_G(u^\prime)\cap N_G(u^{\prime\prime})=V(G)\setminus\{u^{\prime},u^{\prime\prime}\}$ and $x_{u^{\prime}}=x_{u^{\prime\prime}}=1$. In particular,  $G$ contains a copy of $K_{2,n-2}$.

If $k=1$, we assert that $G=K_{2,n-2}=F_{n,1}$, since $G$ is triangle-free and contains $K_{2,n-2}$ as a subgraph. If $k=2$, we shall prove that $u^{\prime}u^{\prime\prime} \in E(G)$. Suppose to the contrary that $u^{\prime}u^{\prime\prime}\notin E(G)$. Then $G[R]$ is $P_3$-free, and contains at most one edge, since otherwise $G$ would contain a copy of $F_2$. Hence, we conclude that $G\cong 2K_1 +(\left(n-4)K_1\cup K_2\right)$ by the maximality of $\rho=\rho(G)$. Let $v_1,v_2$ be the two vertices adjacent in $G[R]$, and let $G^{\prime}=G-v_1v_2+u^{\prime}u^{\prime\prime}$.  Clearly, $G^{\prime}\cong K_2 +(n-2)K_1=F_{n,2}$, which is planar and  $F_2$-free. As $\rho \geq \sqrt{2 n-4}$, by Lemma \ref{lem::8} (ii), we have
$$
\begin{aligned}
\rho(G^\prime)-\rho&\geq \frac{2}{X^TX}(x_{u^{\prime}}x_{u^{\prime\prime}}-x_{v_1}x_{v_2})\\
&\geq \frac{2}{X^TX}\left(1-\left(\frac{2}{\rho}+\frac{4.496}{\rho^{2}}\right)^{2}\right)\\
&\geq \frac{2}{X^TX}\left(1-\left(\frac{2}{\sqrt{2n-4}}+\frac{4.496}{2n-4}\right)^{2}\right)\\
&>0,
\end{aligned}
$$
and hence $\rho(G^\prime)>\rho$, which is impossible by the maximality of $\rho=\rho(G)$. Thus  $u^{\prime}u^{\prime\prime} \in E(G)$, and it follows that $G[R]$ is an empty graph, since otherwise $G$ would contain a copy of $F_2$. Therefore, $G\cong K_2+(n-2)K_1=F_{n,2}$, as desired. 
Now suppose that $k\geq 3$. We have the following claim.

\begin{claim}\label{claim::4.2}
     The vertices $u^\prime$ and $u^{\prime\prime}$ are adjacent. Moreover, $G[R]$ is a disjoint union of some paths, and contains at most $2k-4$ edges.
\end{claim}
\begin{proof}
Note that $G[R]\ncong C_{n-2}$ as $n\geq \max\{2.67\times9^{17},10.2\times2^{2k-4}+2\}$ and $G$ is $F_{k}$-free. By Theorem \ref{thm::1} (ii), $G[R]$ must be a disjoint  union of some paths. Suppose to the contrary that $u^{\prime}u^{\prime\prime} \notin E(G)$. Then we assert that $G[R]$ contains at most $2k-2$ edges. In fact, if $G[R]$ contains at least $2k-1$ edges, then $G[R]$ must have $k$ independent edges because all components of $G[R]$ are paths. 
Thus $G$ contains $k$ edge-disjoint triangles which intersect in the common vertex $u^\prime$, which is impossible because $G$ is $F_k$-free. Denote by $R=\{u_1,u_2,\dots,u_{n-2}\}$. Let $G^{\prime}=G+u^{\prime}u^{\prime\prime}-\sum_{u_iu_j\in E(G)}{u_iu_j}$. Clearly, $G^{\prime}\cong K_2 +(n-2)K_1$, which is planar and $F_k$-free.  By Lemma \ref{lem::8} (ii), 
    $$\begin{aligned}
        \rho(G^{\prime})-\rho &\geq \frac{2}{X^{T} X} \left(x_{u^{\prime}}x_{u^{\prime\prime}}-\sum_{u_iu_j\in E(G)}{x_{u_i}x_{u_j}}\right) \\
        & \geq \frac{2}{X^{T} X} \left( 1-(2k-2)\cdot\left(\frac{2}{\rho}+\frac{4.496}{\rho^{2}}\right)^{2}\right)\\
        &  \geq \frac{2}{X^{T} X} \left( 1-(2k-2)\cdot\left(\frac{2}{\sqrt{2n-4}}+\frac{4.496}{2n-4}\right)^{2}\right)\\
        & >0
    \end{aligned}$$
as $\rho \geq \sqrt{2 n-4}$ and $n\geq \max\{2.67\times9^{17},10.2\times2^{2k-4}+2\}$. Hence, $\rho(G^\prime)>\rho$, contrary to the maximality of $\rho=\rho(G)$. Therefore,  $u^{\prime}u^{\prime\prime} \in E(G)$. In this situation,  we assert that $G[R]$ contains at most $2k-4$ edges. Indeed, if  $G[R]$ contains at least $2k-3$ edges, then $G[R]$ must have $k-1$ independent edges, say $u_1u_2,\ldots,u_{2k-3}u_{2k-2}$. Take $u\in R\setminus\{u_1,\ldots,u_{2k-2}\}$. Then $u^{\prime},u^{\prime\prime},u,u_1,\ldots,u_{2k-2}$ would generate $k$ edge-disjoint triangles which intersect in the common vertex $u^\prime$, and so $G$ contains a copy of $F_k$, a contradiction.
\end{proof}

Let $H^*=P_{2k-3} \cup (n-2k+1)K_1$. Clearly, $K_2+H^*=F_{n,k}$. If $G[R]\cong H^*$, then $G\cong F_{n,k}$, and we are done. If $G[R]\ncong H^*$, by Claim \ref{claim::4.2}, $G[R]$ is a disjoint union of some paths and contains at most $2k-4$ edges, and so we can obtain $H^*$ by applying a series of $(s_1,s_2)$-transformations to $G[R]$,  where $s_2\leq s_1\leq 2k-4$. As $n\geq 10.2\times 2^{2k-4}+2\geq 10.2\times 2^{s_2}+2$, by Lemma \ref{lem::9}, we conclude that $\rho<\rho(K_2+H^*)=\rho(F_{n,k})$, which is impossible by the maximality of $\rho=\rho(G)$. 

This completes the proof.\qed

\subsection{The spectral extremal graph for $(k+1)K_2$-free planar graphs}

In this subsection, we shall prove Theorem \ref{thm::4}, which determines the unique extremal graph in $\mathrm{SPEX}_{\mathcal{P}}(n, (k+1)K_2)$ for every integer  $k\geq 1$, provided that $n$ is sufficiently large relative to $k$. To achieve this goal, we first consider connected  extremal graphs  in 
$\mathrm{SPEX}_{\mathcal{P}}(n, (k+1)K_2)$.

\begin{lem}\label{lem::10}
Let $k$ and $n$ be positive integers with $n\geq\max\{2.67\times9^{17},10.2\times2^{2k-4}+2\}$. Then $M_{n,k}$ is the  unique connected extremal graph in $\mathrm{SPEX}_{\mathcal{P}}(n,(k+1)K_2)$. In particular, $M_{n,k}$ is the unique graph attaining the maximum spectral radius among all connected planar graphs of order $n$ with matching number $k$.
\end{lem}
\begin{proof}
Suppose  $n\geq \max\{2.67\times9^{17},10.2\times2^{2k-4}+2\}$. It is easy to verify that $M_{n,k}$ is planar and $(k+1)K_2$-free for all $k\geq 1$. Suppose that $G$ is a connected extremal graph in $\mathrm{SPEX}_{\mathcal{P}}(n, (k+1)K_2)$. If $k=1$, it is clear that $G=K_{1,n-1}=M_{n,1}$, as desired. If $k\geq2$, then $(k+1)K_2$ is not a subgraph of $K_{2,n-2}$. As $n\geq \max\{2.67\times9^{17},10.2\times2^{2k-4}+2\}>\frac{20}{9}(k+1)$, all conclusions in Theorem \ref{thm::1} hold. By  Theorem \ref{thm::1} (i), there exist two vertices  $u^\prime,u^{\prime\prime}\in V(G)$ such that $R=N_G(u^\prime)\cap N_G(u^{\prime\prime})=V(G)\setminus\{u^{\prime},u^{\prime\prime}\}$ and $x_{u^{\prime}}=x_{u^{\prime\prime}}=1$. In particular,  $G$ contains a copy of $K_{2,n-2}$. If $k=2$, then $G[R]$ is an empty graph because $G$ is $3K_2$-free. Moreover, we assert that $u^\prime u^{\prime\prime}\in E(G)$. Indeed, if $u^\prime u^{\prime\prime}\notin E(G)$, then $G\cong 2K_1 +(n-2)K_1$. Let $G^\prime=G+u^{\prime}u^{\prime\prime}$. Clearly, $G^{\prime}$ is planar and $3K_2$-free. However, $\rho(G^\prime)>\rho$, contrary to the maximality of $\rho=\rho(G)$. Therefore, $u^{\prime}u^{\prime\prime}\in E(G)$, and hence $G\cong K_2 +(n-2)K_1=M_{n,2}$, as desired. Now suppose that $k\geq 3$. We have the following claim.

\begin{claim}\label{claim::4.3}
     The vertices $u^\prime$ and $u^{\prime\prime}$ are adjacent. Moreover, $G[R]$ is a disjoint union of some paths, and contains at most $2k-4$ edges.
\end{claim}
\begin{proof}
Since $n\geq \max\{2.67\times9^{17},10.2\times2^{2k-4}+2\}$ and  $G$ is $(k+1)K_2$-free, we have $G[R]\ncong C_{n-2}$. By Theorem \ref{thm::1} (ii), $G[R]$ must be a disjoint  union of some paths. By contradiction, suppose that $u^{\prime}u^{\prime\prime} \notin E(G)$. Then we assert that $G[R]$ contains at most $2k-2$ edges. In fact, if $G[R]$ contains at least $2k-1$ edges, then $G[R]$ must have $k$ independent edges because all components of $G[R]$ are paths. 
This implies that $(k+1)K_2$ is a subgraph of $G$, and we obtian a contradiction. Denote by $R=\{u_1,u_2,\dots,u_{n-2}\}$. Let $G^{\prime}=G+u^{\prime}u^{\prime\prime}-\sum_{u_iu_j\in E(G)}{u_iu_j}$. Clearly, $G^{\prime}\cong K_2 +(n-2)K_1$, which is planar and $(k+1)K_2$-free.  By Lemma \ref{lem::8} (ii), 
    $$\begin{aligned}
        \rho(G^{\prime})-\rho &\geq \frac{2}{X^{T} X} \left(x_{u^{\prime}}x_{u^{\prime\prime}}-\sum_{u_iu_j\in E(G)}{x_{u_i}x_{u_j}}\right) \\
        & \geq \frac{2}{X^{T} X} \left( 1-(2k-2)\cdot\left(\frac{2}{\rho}+\frac{4.496}{\rho^{2}}\right)^{2}\right)\\
        &  \geq \frac{2}{X^{T} X} \left( 1-(2k-2)\cdot\left(\frac{2}{\sqrt{2n-4}}+\frac{4.496}{2n-4}\right)^{2}\right)\\
        & >0
    \end{aligned}$$
as $\rho \geq \sqrt{2 n-4}$ and $n\geq \max\{2.67\times9^{17},10.2\times2^{2k-4}+2\}$. Hence, $\rho(G^\prime)>\rho$, contrary to the maximality of $\rho=\rho(G)$. Therefore,  $u^{\prime}u^{\prime\prime} \in E(G)$. In this situation,  we assert that $G[R]$ contains at most $2k-4$ edges. Indeed, if  $G[R]$ contains at least $2k-3$ edges, then $G[R]$ must have $k-1$ independent edges, say $u_1u_2,\ldots,u_{2k-3}u_{2k-2}$. Take $u,\Bar{u}\in R\setminus\{u_1,\ldots,u_{2k-2}\}$. Then $u^{\prime},u^{\prime\prime},u,\Bar{u},u_1,\ldots,u_{2k-2}$ would generate $k+1$ independent edges, and so $G$ contains a copy of $(k+1)K_2$, which is a contradiction.
\end{proof}

Let $H^*=P_{2k-3} \cup (n-2k+1)K_1$. Clearly, $K_2+H^*=M_{n,k}$. If $G[R]\cong H^*$, then $G\cong M_{n,k}$, and we are done. If $G[R]\ncong H^*$, by Claim \ref{claim::4.3}, $G[R]$ is a disjoint union of some paths and contains at most $2k-4$ edges, and so we can obtain $H^*$ by applying a series of $(s_1,s_2)$-transformations to $G[R]$,  where $s_2\leq s_1\leq 2k-4$. As $n\geq 10.2\times 2^{2k-4}+2\geq 10.2\times 2^{s_2}+2$, by Lemma \ref{lem::9}, we conclude that $\rho<\rho(K_2+H^*)=\rho(M_{n,k})$, which is impossible by the maximality of $\rho=\rho(G)$.

By the above discussions, we conclude that $M_{n,k}$ is the unique connected extremal graph in $\mathrm{SPEX}_{\mathcal{P}}(n,(k+1)K_2)$ for any $k\geq 1$. Thus the first part of the lemma follows. Also note that the matching number  of $M_{n,k}$ is exactly $k$. Therefore, $M_{n,k}$ is the unique graph attaining the maximum spectral radius among all connected planar graphs of matching number $k$.
\end{proof}

Now we are in a position to give the proof of Theorem \ref{thm::4}.

{\flushleft \it Proof of Theorem \ref{thm::4}.}  
Let $n\geq N+\frac{3}{2}\sqrt{2N-6}$ with $N= \max\{2.67\times9^{17},10.2\times2^{2k-4}+2\}$. Suppose that $G$ is an extremal graph in $\mathrm{SPEX}_{\mathcal{P}}(n, (k+1)K_2)$. If $G$ is connected, by Lemma \ref{lem::10}, we obtain $G\cong M_{n,k}$, and the result follows. Now assume that $G$ is disconnected. Let $G_1,G_2,\ldots,G_{\ell}$ ($\ell\geq 2$) denote the components of $G$. Without loss of generality, we suppose that $\rho(G_1)=\rho(G)=\rho$. For any $i\in\{1,\ldots,\ell\}$, let $V(G_i)=\{u_{i,1},u_{i,2},\dots,u_{i,n_{i}}\}$, and let $k_i$ be the matching number of $G_i$. Clearly, $\sum_{i=1}^{\ell}{k_i} \leq k$ and $\sum_{i=1}^{\ell}{n_i}=n$.  If $k_1\leq k-1$, then $G^{\prime}:=G-\cup_{i=2}^lE(G_i)+ \sum_{i=2}^\ell\sum_{k=1}^{n_i}{u_{i,k}u_{1,1}}$ would be a connected $(k+1)K_2$-free planar graph. As $G_1$ is a proper subgraph of  $G^{\prime}$, we have
$\rho(G^{\prime})>\rho(G_1)=\rho(G)$, which is a contradiction. Therefore, $k_1=k$, and $G=G_1\cup (n-n_1)K_1$. Moreover, we claim that $G_1$ must be an extremal graph in $\mathrm{SPEX}_{\mathcal{P}}(n_1, (k+1)K_2)$, since otherwise $G$ would not be an extremal graph in $\mathrm{SPEX}_{\mathcal{P}}(n, (k+1)K_2)$. If $n_1\geq N=\max\{2.67\times9^{17},10.2\times2^{2k-4}+2\}$, by Lemma \ref{lem::10}, we have $G_1\cong M_{n_1,k}=K_2+(P_{2k-3}\cup (n_1-2k+1)K_1)$, and hence $G\cong M_{n_1,k}\cup (n-n_1)K_1$. In this situation, $G$ is a proper subgraph of $M_{n,k}=K_2+(P_{2k-3}\cup (n-2k+1)K_1)$, and it follows that $\rho=\rho(G)<\rho(M_{n,k})$, contrary to the maximality of $\rho=\rho(G)$.  Thus it suffices to consider the case that  $n_1<N=\max\{2.67\times9^{17},10.2\times2^{2k-4}+2\}$. If $k\geq 2$, by Lemma \ref{lem::1}, we obtain
$$\rho=\rho(G_1)\leq 2+\sqrt{2n_{1}-6}<  2+\sqrt{2N-6}\leq \frac{1+\sqrt{8n-15}}{2}=\rho(K_2+(n-2)K_1)\leq \rho(M_{n,k})$$
as $n\geq N+\frac{3}{2}\sqrt{2N-6}$ and $K_2+(n-2)K_1$ is a subgraph of $M_{n,k}$.  However, this is  impossible by the maximality of $\rho=\rho(G)$. If $k=1$, then $G$ must be a proper subgraph of $K_{1,n-1}=M_{n,1}$, and we obtain a contradiction. Therefore, we conclude that $G\cong M_{n,k}$, and the result follows.
\qed

\section*{Acknowledgement}

X. Huang is supported by the National Natural Science Foundation of China (Grant No. 11901540). Huiqiu Lin was supported by the National Natural Science Foundation of China (Grant No. 12271162), Natural Science Foundation of Shanghai (Nos. 22ZR1416300 and 23JC1401500) and the Program for Professor of Special Appointment (Eastern Scholar) at Shanghai Institutions of Higher Learning (No. TP2022031).

\end{document}